\documentclass[11pt,reqno]{amsart}
\usepackage{amsmath,amssymb, mathrsfs, stmaryrd, comment}
\usepackage{graphicx, color, appendix}
\usepackage{hyperref}

\newcommand{\pl}{\partial}

\newcommand{\rw}{\rightarrow}

\newcommand{\R}{\mathbb{R}}

\newcommand{\lt}{\left}
\newcommand{\rt}{\right}
\newcommand{\na}{\mathring{\nabla}}
\newcommand{\cb}{\underline{c}}
\newcommand{\Fb}{\underline{F}}

\newtheorem{theorem}{Theorem}
\newtheorem{proposition}[theorem]{Proposition}

\newtheorem{lemma}[theorem]{Lemma}

\theoremstyle{definition}
\newtheorem{definition}[theorem]{Definition}
\newtheorem{remark}[theorem]{Remark}

\begin{document} 

\title[Angular Momentum and Center-of-Mass at Null Infinity]{Evaluating Quasi-local Angular Momentum and Center-of-Mass at Null Infinity}

\author{Jordan Keller}
\address{Jordan Keller\\
Black Hole Initiative\\
Harvard University, USA}
\email{jordan\_keller@fas.harvard.edu}

\author{Ye-Kai Wang}
\address{Ye-Kai Wang\\Department of Mathematics\\National Cheng Kung University, Taiwan}
\email{ykwang@mail.ncku.edu.tw}

\author{Shing-Tung Yau}
\address{Shing-Tung Yau\\Department of Mathematics\\Harvard University, USA}
\email{yau@math.harvard.edu}

\thanks{Portions of this work were carried out during a visit to the Yau Mathematical Sciences Center at Tsinghua University.  The first author wishes to thank this institution for its hospitality.  In addition, the same author thanks the John Templeton Foundation for its support.  The authors wish to thank Po-Ning Chen and Mu-Tao Wang for helpful comments and stimulating conversations related to this work. S.-T. Yau is partially supported by NSF grant DMS-160781 and Simons Foundation 385581. Y.-K. Wang is supported by Taiwan MOTS grants 105-2115-M-006-016-MY2 and 107-2115-M-006-001-MY2.}

\begin{abstract}
We calculate the limits of the quasi-local angular momentum and center-of-mass defined by Chen-Wang-Yau \cite{CWY} for a family of spacelike two-spheres approaching future null infinity in an asymptotically flat spacetime admitting a Bondi-Sachs expansion.  Our result complements earlier work of Chen-Wang-Yau \cite{CWY_null}, where the authors calculate the limits of the quasi-local energy and linear momentum at null infinity.  Finiteness of the center-of-mass limit requires that the spacetime be in the so-called center-of-mass frame, a mild assumption on the mass aspect function amounting to vanishing of linear momentum at null infinity.  With this condition and the assumption that the Bondi mass is non-trivial, we obtain explicit expressions for the angular momentum and center-of-mass at future null infinity in terms of the observables appearing in the Bondi-Sachs expansion of the spacetime metric. 
\end{abstract}

\maketitle

\section{Introduction}
\subsection{}
In general relativity, the gravitational fields of isolated systems are modeled by asymptotically flat spacetimes.  Such systems emit gravitational waves \cite{LIGO} which travel at the speed of light and eventually reach future null infinity of the spacetime.  In the study of gravitational waves, it is desirable to find a suitable notion of conserved quantities defined at null infinity. While we have the well-accepted Bondi-Sachs energy-momentum, there is no such consensus for angular momentum or center-of-mass integral. To shed light on this problem, we calculate the limits of the quasi-local angular momentum and center-of-mass integral defined by Chen-Wang-Yau \cite{CWY} for a family of spacelike two-spheres approaching future null infinity in an asymptotically flat spacetime $N$ admitting a Bondi-Sachs coordinate system. We start this introduction by briefly reviewing the Bondi-Sachs coordinates.

Concurrent work of Bondi, van der Burg, and Metzner \cite{BVM} and Sachs \cite{Sachs} introduced the Bondi-Sachs coordinates in order to clarify the nature of gravitational radiation. The coordinate system is built on the geometry of outgoing null hypersurfaces. In particular, the luminosity distance of null geodesics $r$ is taken as a coordinate function. Let $N$ be a vacuum spacetime with metric given in Bondi-Sachs coordinates $(u,r,x^{A})$ by
\begin{align*}
-UVdu^2 - 2Ududr+\sigma_{AB}(dx^{A}+W^{A}du)(dx^{B}+W^{B}du)
\end{align*}
where we demand 
\begin{align*}
\det \left(\frac{\sigma_{AB}}{r^2}\right) = \det(\mathring\sigma_{AB}),
\end{align*}
where $\mathring\sigma_{AB}$ is the round metric on the unit sphere. Denote the level sets of $u$ by $\mathcal{C}_u$. Assuming that the metric coefficients can be expanded into power series in $\frac{1}{r}$, the above determinant condition implies 
\[ \sigma_{AB} = r^2 \mathring\sigma_{AB} + r C_{AB} + O(1),\]
where $C_{AB}(u, x^D)$ is a symmetric traceless two-tensor on $S^2$, referred to as the shear tensor.  Moreover, the null vacuum constraint equations enjoy a remarkable hierarchy that allows us to solve the metric coefficients on $\mathcal{C}_u$ order by order, with $C_{AB}$ being free data. In particular,
\begin{align*}
\begin{split}
U &= 1-\frac{1}{16r^2}\left(C_{DE}C^{DE}\right) + O(r^{-3}),\\
V & = 1 - \frac{2m}{r} + O(r^{-2}),\\
W^{A} &= \frac{1}{2r^2}\mathring\nabla^{D}C_{D}^{A} + \frac{1}{r^3}\left(\frac{2}{3}N^{A} - \frac{1}{16}\mathring\nabla^{A}\left(C_{DE}C^{DE}\right) - \frac{1}{2}C^{A}_{B}\mathring\nabla^{D}C_{D}^{B}\right) \\
&\quad + O(r^{-4}).
\end{split}
\end{align*} 
Here and hereafter, tensor contraction is performed with respect to $\mathring\sigma_{AB}$, unless otherwise noted.  The function $m(u,x^A)$ and the spherical vector $N^A(u,x^D)$ are referred to as the mass aspect and the angular momentum aspect, respectively.  By comparison with static solutions, the Bondi-Sachs energy-momentum 4-vector $(e, p^1, p^2, p^3)$ is defined by
\begin{align}
e &= \frac{1}{4\pi}\int_{S^2} m, \label{BondiMass}\\
p^i &= \frac{1}{4\pi}\int_{S^2} m \tilde X^i, \label{LinearMomentum}
\end{align}
where $\tilde X^i$ are the first eigenfunctions on $(S^2, \mathring\sigma)$.  The positive mass theorem \cite{Schoen-Yau, Horowitz-Perry} asserts that the Bondi-Sachs energy-momentum 4-vector is future-directed timelike if there is a complete spacelike hypersurface intersecting null infinity in the given cut such that the dominant energy condition is satisfied and has nonflat domain of dependence. In particular, $e > 0$.

\subsection{}
To put the Bondi-Sachs energy-momentum on physical grounds and to seek the correct definitions of angular momentum, it is necessary to understand the symmetries of null infinity $\mathscr{I}^+$: the BMS group and its Lie algebra the BMS algebra. The BMS algebra consists of vector fields
\begin{align*}
(f + uY^1 )\frac{\pl}{\pl u} + Y^A \frac{\pl}{\pl x^A}   
\end{align*}
 on $\mathscr{I}^+ = (-\infty,\infty) \times S^2$ where $f$ is an arbitrary function on $S^2$ and $Y^A \frac{\pl}{\pl x^A}$ is a conformal Killing vector field on $S^2$, $\na^A Y^B + \na^B Y^A = 2Y^1 \mathring\sigma_{AB}$. 
The BMS algebra is similar to the Poincar\'{e} algebra but contains an infinite-dimensional abelian subalgebra, the {\it supertranslations}
$f \frac{\pl}{\pl u} $
instead of the 4-dimensional translations where $f$ is taken to be mode 0 and 1 spherical harmonics. The quotient of the supertranslations is again isomorphic to the Lorentz algebra. The fact that the Lorentz algebra sits in the BMS algebra in infinitely dimensional ways by conjugation with supertranslations, referred as {\it supertranslation ambiguity} in physics literature,  is the major impediment of defining angular momentum at null infinity that transforms covariantly as in special relativity. For supertranslation ambiguity, we refer to \cite{ADK, Moreschi04} for an explanation.

The first definition of conserved quantities with respect to the full BMS algebra was the Winicour-Tamburino \cite{Winicour} linkage. They consider a propagation law of the BMS vector fields and define the linkage as a modification of Komar's integral, see \cite[Section V]{Winicour}. The expression in terms of Bondi-Sachs data is given in \cite[(8.16)]{Winicour}:  
\begin{align*}
\frac{1}{8\pi} \int_{S^2} Y^A (N_A - \frac{1}{4}C_{AB} \na_D C^{BD} - \frac{1}{16} \na_A (C_{DE}C^{DE}) ) + f (m+ \frac{1}{4} \na^A\na^B C_{AB}). 
\end{align*}
Other early proposals are Bramson \cite{Bramson} and Prior \cite{Prior}. 

The turning point of the story is Ashtekar-Strubel \cite{Ashtekar-Streubel}. They showed that the BMS group induced canonical transformations on the phase space and the corresponding Hamiltonians can be interpreted as fluxes of conserved quantities. This is the first definition that vanishes when there is no gravitational radiations \cite{Ashtekar-Winicour}. Next, by extending Penrose's definition of quasi-local angular momentum based on twistor method \cite{Pen_quasilocal}, Dray-Strubel \cite{Dray-Streubel} defined a new set of conserved quantities for the full BMS algebra. The follow-up independent works Dray \cite{Dray} and Shaw \cite{Shaw} showed that the flux of Dray-Strubel conserved quantities is the Ashtekar-Streubel flux. Finally, Wald-Zoupas \cite{Wald-Zoupas} recovered the Dray-Strubel conserved quantities via the Hamiltonian framework, thus providing them a physical interpretation. We recommend Section III of Flanagan-Nichols \cite{Flanagan-Nichols} for an excellent exposition. In particular, one finds in \cite[(3.5)]{Flanagan-Nichols} the definition of Dray-Streubel conserved quantities in terms of Bondi-Sachs data 
\begin{align*}
\frac{1}{8\pi} \int_{S^2} Y^A ( N_A  - \frac{1}{4} C_{AB} \na_D C^{BD} - \frac{1}{16} \na_A (C_{DE}C^{DE}) ) + (f + uY^1)(2m).
\end{align*}

To close this subsection, we mention two important related works. First, under the framework for stability of Minkowski spacetime introduced by Christodoulou-Klainerman \cite{CK}, Rizzi proposed a definition of angular momentum at null infinity \cite{R_thesis, R_def}. Second, Chru\`{s}ciel-Jezierski-Kijowski \cite{CJK} developed a Hamiltonian formalism for a description of field theories in radiation regime and obtained a definition of angular momentum and center-of-mass integral at null infinity. In terms of Bondi-Sachs data, they are given by \cite[(6.117)]{CJK}
\begin{align*}
\frac{1}{8\pi} \int_{S^2} Y^A ( N_A - \frac{1}{4} C_{AB} \na_D C^{BD} - \frac{1}{16} \na_A (C_{DE}C^{DE}) ).
\end{align*}

\subsection{}
In this work, we study the limits of Chen-Wang-Yau quasi-local angular momentum and center-of-mass integral \cite{CWY}. The definition, reviewed in Section 2.2, applies to arbitrary spacelike 2-surface $\Sigma$ with spacelike mean curvature vector. It is based on the surface Hamiltonian extracted from the Hilbert-Einstein gravitation action. The novelty of Chen-Wang-Yau definition is that it involves a reference term whose value depends on the isometric embedding of $\Sigma$ into the Minkowski spacetime. Solving the {\it optimal isometric embedding equation} requires the (Hodge) decomposition of the shear tensor $C_{AB} = F_{AB} + \Fb_{AB}$ where 
\begin{align*}
F_{AB} &= \na_A \na_B c - \frac{1}{2} \mathring\Delta c \mathring\sigma_{AB} \\
\Fb_{AB} &= \frac{1}{2} ( \mathring\epsilon_{AD} \na_B \na^D \cb + \mathring\epsilon_{BD} \na_A \na^D \cb). 
\end{align*}
 
Our main results, Theorems \ref{main_angmom} and \ref{main_com}, state the following:
\begin{theorem}
Fix a null hypersurface $\{ u=u_0 \}$ and denote the level set of $r$ on $\{ u=u_0\}$ by $\Sigma_r$. Assuming that $\{ u=u_0\}$ has vanishing linear momentum, then the limits of  angular momentum $J^k$ and center-of-mass integral $C^i$ of $\Sigma_r$ are given by  
\begin{align*}
J^{k} = \frac{1}{8\pi}\int_{S^2}& \mathring\epsilon^{AB} \na_B \tilde X^k \lt( N_{A} -\frac{1}{4}C_{AB}\mathring\nabla_{D}C^{DB} - c\mathring\nabla_{A}m   \rt) .
\end{align*}
and
\begin{align*}
\begin{split}
C^i &= \frac{1}{8\pi e} \int_{S^2} \mathring\nabla^{A}\tilde{X}^{i} \Bigg[ N_A - \frac{1}{4} C_{AB} \na_D C^{DB} - \frac{1}{16} \na_A \lt( C^{DE}C_{DE}\rt) \\
&\qquad\qquad\qquad\qquad\qquad\qquad\qquad\qquad - c \na_A m  + 2 \mathring\epsilon_{AB} (\na^B \cb) m \Bigg]\\
&\quad  + \frac{1}{8\pi e} \int_{S^2}  3 \tilde X^i cm - \frac{1}{4} \tilde X^i \na_A \Fb^{AB} \na^D \Fb_{DB},
\end{split}
\end{align*}
where $\mathring\epsilon_{AB}$ denotes the area form with respect to $\mathring\sigma_{AB}$.
\end{theorem}

Comparing with Dray-Streubel's definition, Chen-Wang-Yau's definition contains correction terms that result from the optimal isometric embedding equations. In subsequent work joint with Po-Ning Chen and Mu-Tao Wang \cite{CKWWY}, we show that Chen-Wang-Yau's angular momentum and center-of-mass integral are distinguished by their supertranslation invariance of total flux.
 
\begin{figure}[t]
\caption{The null hypersurface $\{ u=u_0\}$ and the spherical section $\Sigma_r$}
\centering
\includegraphics[scale=1]{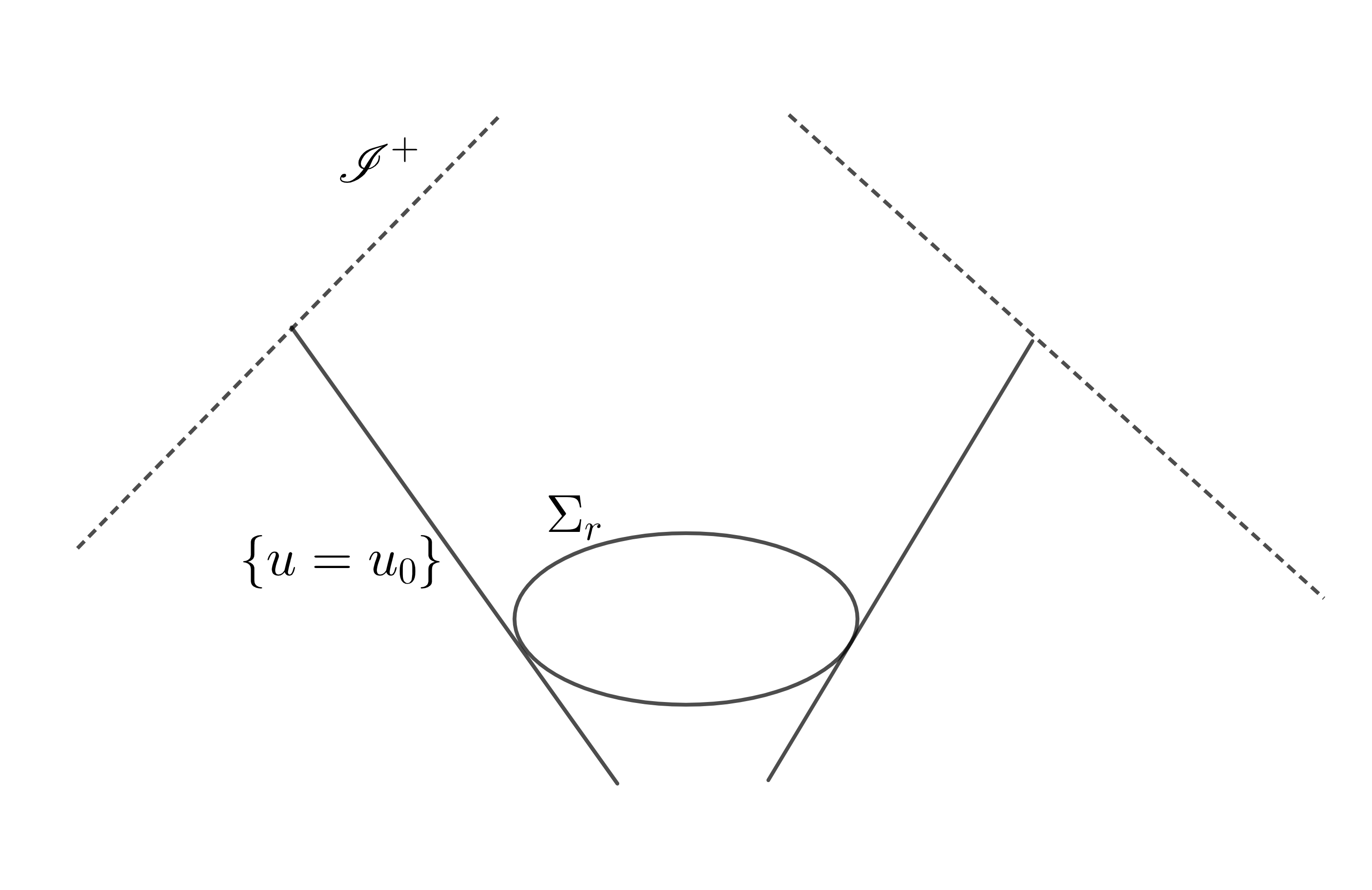}\end{figure} 

The rest of the paper is organized as follows. In Section 2, we expand the metric coefficients in Bondi-Sachs coordinates to the necessary order and review the definition of Chen-Wang-Yau conserved quantities. In Sections 3 and 4, we compute the physical data and the reference data in Chen-Wang-Yau definitions respectively. In Section 5, we discuss the center-of-mass-frame and solve the optimal isometric embedding equation. The angular momentum and center-of-mass integral are evaluated in Sections 6 and 7 respectively. In Section 8, it is shown that in Kerr spacetime, we recover the expected values of angular momentum and center-of-mass integral. There are two appendices on useful identities and facts of symmetric traceless 2-tensors on $S^2$.

\section{Preliminaries}
\subsection{Bondi-Sachs Coordinates and Asymptotic Expansion}

Let $N$ be a spacetime with metric $g$ given in Bondi-Sachs coordinates $(u,r,x^{A})$ by
\begin{equation}\label{CWYform}
g = -UVdu^2 - 2Ududr+\sigma_{AB}(dx^{A}+W^{A}du)(dx^{B}+W^{B}du)
\end{equation}
and satisfy the {\it determinant condition}
\begin{equation}\label{areaNorm}
\det \left(\frac{\sigma_{AB}}{r^2}\right) = \det(\mathring\sigma_{AB}),
\end{equation}
where $\mathring\sigma_{AB}$ is the round metric on the unit sphere.  In this way, the spacetime metric above has exactly six degrees of freedom.  For future reference, we denote the covariant derivatives associated with $\sigma_{AB}$ and $\mathring\sigma_{AB}$ by $\nabla_{A}$ and $\mathring\nabla_{A}$, respectively.  We also note that, for vector fields $X^{A}$ on the sphere, we have 
\begin{equation}\label{divergenceRelation}
 div_{\sigma} X^{A} = div_{\mathring\sigma} X^{A}
 \end{equation}
 owing to the determinant condition \eqref{areaNorm}.

Assuming a Bondi-Sachs expansion from null infinity in powers of $\frac{1}{r}$, we have
\begin{align}\label{CWYexpansion}
\begin{split} 
U &= 1-\frac{1}{16r^2}\left(C_{DE}C^{DE}\right) + O(r^{-3}),\\
V &=1-\frac{2m}{r} + O(r^{-2}),\\
W^{A} &= \frac{1}{r^2} W^{(-2)A} + \frac{1}{r^3} W^{(-3)A} + O(r^{-4}) \\
&= \frac{1}{2r^2}\mathring\nabla^{D}C_{D}^{A} + \frac{1}{r^3}\left(\frac{2}{3}N^{A} - \frac{1}{16}\mathring\nabla^{A}\left(C_{DE}C^{DE}\right) - \frac{1}{2}C^{A}_{B}\mathring\nabla^{D}C_{D}^{B}\right)\\
&\quad + O(r^{-4}).
\end{split}
\end{align}
The function $m(u,x^{A})$ and the spherical vector $N^A(u,x^{D})$ are referred to as the mass aspect and the angular momentum aspect, respectively.  The symmetric traceless spherical two-tensor $C_{AB}(u,x^{D})$ is referred to as the shear tensor.

The expansion above is derived from the Einstein equations expressed in the Bondi-Sachs variables.  In particular, there are a hierarchy of initial data equations which can be systematically integrated to yield higher terms in the metric expansion, assuming expressions in powers of $\frac{1}{r}$ as mentioned at the outset.  See M\"{a}dler-Winicour \cite{MW} for details.

Moreover, we need a further term in the expansion of $V$ in calculating the quasi-local quantities. The result is well-known. See \cite[(3.5)]{CJM} for example. For completeness, we give a derivation in the appendix.
\begin{proposition}\label{V}
The metric function $V$ expands as
\begin{align}
V = 1 - \frac{2m}{r} + \frac{1}{r^2} V^{(-2)} + O(r^{-3})
\end{align}
where
\begin{align}\label{Vexpansion}
V^{(-2)} = \frac{1}{4} \na_A C^{AB} \na^D C_{BD} + \frac{1}{16} C_{DE}C^{DE} + \frac{1}{3} \na^A N_A .
\end{align}
\end{proposition}

\begin{remark}\label{NA_convention}
We compare different conventions on angular momentum aspect in Hawking-Perry-Strominger \cite[(2.2)]{HPS},  Chrusicel-Jezierski-Kijowski \cite[(5.99)]{CJK}, Barnich-Troessaert \cite[(4.37)]{Bar-Tro} and M\"{a}dler-Winicour \cite[(56)]{MW}. Note that these authors use $U^A = -W^A$ in their papers. Our choice features that $\lim_{r\rw\infty} R_{Arru} = N_A$.  
\begin{enumerate}
\item $N^A_{(HPS)} = N^A - u \na^A m$ 
\item $N^A_{(CJK)} = -\frac{1}{3}N^A$ 
\item $N^A_{(BT)} = N^A - \frac{3}{16} C_{BD}\na^A C^{BD} -\frac{1}{4} C^{AB} \na^D C_{DB}$ 
\item $N^A_{(MW)} = -\frac{1}{3} N^A_{(BT)}$
\end{enumerate} 
\end{remark}

\subsection{Chen-Wang-Yau quasi-local conserved quantities}
Next, we describe the quasi-local quantities to be considered, per Chen-Wang-Yau \cite{CWY}.  Let $(\Sigma, \sigma)$ be a closed embedded spacelike two-sphere in a spacetime $(N, g)$, with spacelike mean curvature vector $H$.  The physical data used in the quasi-local definition consists of the triple $(\sigma, |H|, \alpha_{H})$, where $\sigma$ is the induced metric of $\Sigma$, $|H|$ is the norm of its mean curvature vector, and $\alpha_{H}$ is the connection one-form of the normal bundle with respect to the mean curvature vector; that is,
\[\alpha_{H}(\cdot) = g(\nabla^{N}_{(\cdot)}\frac{J}{|H|}, \frac{H}{|H|}),\]
where $J$ is the future-directed timelike vector obtain via reflection of $H$ through the incoming light cone in the normal bundle.

The reference data used in the quasi-local definition is specified with respect to an isometric embedding $X : (\Sigma,\sigma) \hookrightarrow (\R^{3,1}, \eta)$ of $\Sigma$ into the Minkowski spacetime.  In terms of the image surface $X(\Sigma)$ in $\R^{3,1}$, we have the reference data $(|H_0|, \alpha_{H_0})$, analogous to the physical data.

Let $(t,x^1, x^2, x^3)$ be a Cartesian coordinate system on the Minkowski spacetime $(\R^{3,1},\eta)$.  We take $T_0$ to be a future-directed, unit timelike vector field in the Minkowski spacetime, interpreted as an observer, and split $X = (X^0, X^1, X^2, X^3)$ into temporal and spatial components.  We define the height function 
\begin{equation}\label{tauDef}
\tau := -\langle T_0, X \rangle
\end{equation}
in terms of $T_0$ and $X$.  Further, we define the density function
\begin{equation}\label{densityDef}
\rho := \frac{\sqrt{|H_0|^2 + \frac{(\Delta \tau)^2}{1+|\nabla \tau|^2}} - \sqrt{|H|^2 + \frac{(\Delta \tau)^2}{1+|\nabla \tau|^2}}}{\sqrt{1+|\nabla\tau|^2}}
\end{equation}
and the current one-form
\begin{equation}\label{currentDef}
j := \rho\nabla \tau -\nabla\left[\sinh^{-1}\left(\frac{\rho\Delta\tau}{|H_0||H|}\right)\right] -\alpha_{H_0} + \alpha_{H}
\end{equation}
on $\Sigma$. Note that we identify $\Sigma$ and $X(\Sigma)$ implicitly by the isometric embedding $X.$ In particular, $|H_0|, \alpha_{H_0}$ and $\tau$ stand for $|H_0| \circ X, X^*(\alpha_{H_0})$ and $\tau \circ X$ in the definitions \eqref{densityDef} and \eqref{currentDef}.
 
There are many such choices of isometric embedding $X$ and observer $T_0$ in the Minkowski spacetime.  Per Chen-Wang-Yau \cite{CWY}, we consider only those pairs $(X, T_0)$ with associated data satisfying the optimal isometric embedding equation
\begin{equation}\label{oiee}
div_{\sigma}j = 0.
\end{equation}

The quasi-local angular momentum and center-of-mass integral are defined with respect to rotation and boost Killing fields in Minkowski spacetime, respectively the images of the Lorentz Killing vector fields
\begin{equation}
K_{i,j} := x^i\partial_{j} - x^j\partial_{i},\ i, j = 1, 2, 3;\  i < j
\end{equation}
and \begin{equation}
K_i := x^i\partial_{t} + t\partial_{i},\ i = 1,2,3
\end{equation}
under Lorentz transformations.

Per Chen-Wang-Yau \cite{CWY}, the quasi-local center-of-mass and angular momentum are defined as follows:
\begin{definition}\label{qlCharge}
Given a surface $(\Sigma, \sigma)$ in $(N,g)$, suppose that the pair $(X, T_0)$ provides an isometric embedding of $\Sigma$ into Minkowski spacetime such that the optimal isometric embedding equation \eqref{oiee} is satisfied.  Writing $T_0 = A(\partial_{t}) = A((1,0,0,0))$ for a Lorentz transformation $A$, we define the components of the quasi-local angular momentum by
\begin{equation}\label{qlAM}
J^k(\Sigma,X,T_0)\\ =-\frac{\epsilon^{ijk}}{8\pi}\int_{\Sigma}[\langle A(K_{i,j}), T_0 \rangle \rho + (A(K_{i,j}))^{T}\cdot j]d\Sigma
\end{equation}
and the components of the quasi-local center-of-mass integral by
\begin{equation}\label{qlCOM}
C^i(\Sigma,X,T_0)\\ =-\frac{1}{8\pi e}\int_{\Sigma}[\langle A(K_{i}), T_0 \rangle \rho + (A(K_{i}))^{T}\cdot j]d\Sigma,
\end{equation}
where $\langle \cdot, \cdot \rangle$ denotes the Minkowskian inner product and $K^{T}$ denotes the projection of a Lorentz Killing field $K$ onto the tangent space of the image $X(\Sigma)$, such that
\begin{equation}\label{projectionFormula}
K^{T} = \langle K, \partial_{A}X \rangle \sigma^{AC}\partial_{C}.
\end{equation}
In the angular momentum expression, we make use of the volume form $\epsilon_{ijk}$ of $\R^3$ written with respect to the coordinates $(x^1, x^2, x^3)$, such that $k = 1, 2,3; k \neq i, j$.
\end{definition}

\section{The Physical Data}
We fix a null hypersurface $\{ u = u_0\}$ and study the surfaces $(\Sigma_{r}, \sigma_{r})$ with constant luminosity distance in the null hypersurface. Per Chen-Wang-Yau \cite{CWY_null} and the Bondi-Sachs expansion \eqref{CWYexpansion}, the physical data is specified by 
\begin{equation}\label{metricExpansion}
\sigma_{AB,r} = r^2\mathring\sigma_{AB} + rC_{AB} + \frac{1}{4}(C_{DE}C^{DE})\mathring\sigma_{AB} + O(r^{-1}),
\end{equation}
\begin{equation}\label{meanCurv}
|H_{r}|^2 = \frac{4}{Ur}\left(\frac{V}{r}+div_{\sigma}W\right),
\end{equation}
\begin{align}\label{connectionForm}
\begin{split}
g(\nabla^{N}_{\partial_{A}}J_{r},H_{r}) &= -\frac{2}{rU}\partial_{A}\left(\frac{V}{r}+div_{\sigma} W\right)\\
&+ \frac{2}{rU^2}\left(\frac{V}{r}+div_{\sigma} W\right)\sigma_{AC}\partial_{r}W^{C}.
\end{split}
\end{align}

We rewrite \eqref{meanCurv} and \eqref{connectionForm} in terms of the Bondi-Sachs expansion \eqref{CWYexpansion} in the following proposition:

\begin{proposition}
With respect to the Bondi-Sachs expansion \eqref{CWYexpansion}, the norm of the mean curvature \eqref{meanCurv} expands as
\begin{align}\label{H}
\begin{split}
&|H_{r}| = \frac{1}{r}|H|^{(-1)} + \frac{1}{r^2}|H|^{(-2)} + \frac{1}{r^3}|H|^{(-3)} + O(r^{-4})\\
&= \frac{2}{r} - \frac{2m}{r^2} + \frac{1}{2r^2}\left(\mathring\nabla^{D}\mathring\nabla^{E}C_{DE}\right) \\
&\quad + \frac{1}{4r^3}\na_A C^{AB} \na^D C_{BD} + \frac{1}{8r^3}(C_{DE}C^{DE}) + \frac{1}{r^3} \na^E N_E  \\ &\quad - \frac{1}{16r^3} \mathring\Delta (C_{DE}C^{DE}) - \frac{1}{2} \na_A \lt( C^A_B \na^D C^B_D \rt)\\
&\quad - \frac{m^2}{r^3} + \frac{1}{2r^3} m \lt( \na^D\na^E C_{DE} \rt) - \frac{1}{16r^3} \lt( \na^D\na^E C_{DE}\rt)^2 + O(r^{-4}),
\end{split}
\end{align}
and the connection one-form \eqref{connectionForm} expands as
\begin{align}\label{alphaH}
\begin{split}
&\alpha_{H,r}(\partial_{B}) =  \frac{1}{r}\alpha_{H}^{(-1)}(\partial_{B}) + \frac{1}{r^2}\alpha_{H}^{(-2)}(\partial_{B}) + O(r^{-3})\\
&=\frac{1}{r}\mathring\nabla_{B}m - \frac{1}{4r}\mathring\nabla_{B}\left(\mathring\nabla^{D}\mathring\nabla^{E}C_{DE}\right) - \frac{1}{2r}\mathring\nabla^{A}C_{AB}\\
&\quad + \frac{1}{4r^2}C_{BD} \na_E C^{DE} - \frac{1}{r^2}N_B + \frac{3}{32r^2}\na_B(C_{DE}C^{DE}) +\frac{1}{r^2} \na_B \mathfrak{f} + O(r^{-3}).
\end{split}
\end{align}
Here the one-form $\na_B \mathfrak{f}$ comes from the $\frac{1}{r^2}$ order of $-\frac{1}{2}\pl_B \log \lt( \frac{V}{r} + div_\sigma W \rt)$:
\begin{align}
\begin{split}
\mathfrak{f} &= -\frac{1}{2} \na_A N^A + \frac{1}{32} \mathring\Delta (C_{DE}C^{DE}) - \frac{1}{8}\na_A C^{AB} \na^D C_{BD}\\
& - \frac{1}{32}C_{DE}C^{DE} +\frac{1}{4} \lt( 2m - \frac{1}{2}\na^D\na^E C_{DE} \rt)^2 +\frac{1}{4}\mathring\nabla_A(C^{AD}\mathring\nabla^{E}C_{ED}).
\end{split}
\end{align}  
\end{proposition}

\begin{proof}
From \eqref{meanCurv} and \eqref{CWYexpansion}, we get
\begin{align*}
|H_r|^{(-2)} &= -2m + \na_A W^{A(-2)},\\
|H_r|^{(-3)} &= V^{(-2)} + \na_A W^{A(-3)} + \frac{1}{16} C_{DE}C^{DE} - m^2 + m \na_A W^{A(-2)} \\
&\quad - \frac{1}{4} (\na_A W^{A(-2)})^2.
\end{align*} 
Plugging in \eqref{CWYexpansion}, and \eqref{Vexpansion}, we obtain \eqref{H}.

Regarding the connection one-form, we have from \eqref{CWYexpansion}
\begin{align*}
\alpha_{H,r}(\pl_B) &= -\frac{1}{2} \pl_B \log \lt( \frac{V}{r} + div_\sigma W \rt) + \frac{1}{2U} \sigma_{BD} \pl_r W^D \\
&=-\frac{1}{2} \pl_B \log \lt( 1 - \frac{2m}{r} + \frac{1}{r} \na_D W^{D(-2)} + \frac{1}{r^2}V^{(-2)} + \frac{1}{r^2} \na_D W^{D(-3)}\rt)\\
&\quad - \frac{1}{r}W^{(-2)}_B - \frac{1}{r^2} C_{BD} W^{D(-2)} - \frac{3}{2r^2} \mathring\sigma_{BD} W^{D(-3)}.
\end{align*}
Plugging in \eqref{CWYexpansion}, and \eqref{Vexpansion}, we obtain \eqref{alphaH}.

\end{proof}
\section{The Reference Data}

We consider isometric embeddings $X_r : (\Sigma_r,\sigma_r) \hookrightarrow (\R^{3,1}, \eta)$, with components $X_{r} = (X^0_{r}, X^1_{r}, X^2_{r}, X^3_r)$ expanding according to the ansatz
\begin{align}\label{embeddingExpansion}
\begin{split}
X^0_{r} &= X^{0(0)} + \frac{1}{r}X^{0(-1)} + O(r^{-2}),\\
X^i_{r} &= r\tilde{X}^i + X^{i(0)} + \frac{1}{r}X^{i(-1)} + O(r^{-2}),
\end{split}
\end{align}
where $\tilde{X} = (\tilde{X}^1,\tilde{X}^2,\tilde{X}^3)$ are the standard coordinate functions on the unit sphere. The leading order of $X^i_r$ is chosen according to the fact that $\Sigma_r$ is asymptotically a round sphere of radius $r$. The isometric embedding equation
\begin{equation}\label{IE}
dX^i_{r,A} \cdot dX^i_{r,B} = \sigma_{AB,r} + dX^0_{r,A}\cdot dX^0_{r,B},
\end{equation}
and the metric expansion \eqref{metricExpansion} imply the linearized equation
\begin{equation}\label{linIE}
d\tilde{X}^i_{A} \cdot dX^{i(0)}_{B} + dX^{i(0)}_{A} \cdot d\tilde{X}^i_{B} = C_{AB}.
\end{equation}

Recall from the introduction that the shear tensor $C_{AB}(u,x^{D})$ is symmetric and traceless; it is well-known that $C_{AB}(u,x^{D})$ admits the decomposition 
\begin{align}\label{CPotentials}
\begin{split}
C_{AB}(u,x^{A}) &= \left(\mathring\nabla_{A}\mathring\nabla_{B} - \frac{1}{2}\mathring\sigma_{AB}\mathring\Delta\right)c(u,x^{D})\\
&+ \frac{1}{2}\left(\mathring{\epsilon}_{AD}\mathring\nabla^{D}\mathring\nabla_{B} + \mathring{\epsilon}_{BD}\mathring\nabla^{D}\mathring\nabla_{A}\right)\underline{c}(u,x^{D}),
\end{split}
\end{align}
with scalar potentials $c(u,x^{D})$ and $\underline{c}(u,x^{D})$ and $\mathring\epsilon_{AB}$ the area form of the standard unit sphere. See Appendix B for a proof. Without loss of generality, we can assume that the potentials $c$ and $\underline{c}$ have spherical harmonic expansions with support in $\ell \geq 2$.  

As a shorthand, we write
\begin{align}
F_{AB}(u,x^{A}) &:= \left(\mathring\nabla_{A}\mathring\nabla_{B} - \frac{1}{2}\mathring\sigma_{AB}\mathring\Delta\right)c(u,x^{A}),\\
\underline{F}_{AB}(u,x^{A}) &:= \frac{1}{2}\left(\mathring{\epsilon}_{AD}\mathring\nabla^{D}\mathring\nabla_{B} + \mathring{\epsilon}_{BD}\mathring\nabla^{D}\mathring\nabla_{A}\right)\underline{c}(u,x^{A}).
\end{align}

We respectively refer to $F_{AB}$ and $\underline{F}_{AB}$ as the closed and co-closed components of $C_{AB}$, with closed and co-closed potentials $c$ and $\underline{c}$. In physics literature, they are called polar and axial parts. Our denomination comes from the fact that the divergence operator is an isomorphism from the closed (co-closed) symmetric traceless 2-tensors to closed (co-closed) 1-forms on $S^2$. See \cite[Proposition 2.4]{CKWWY}.

With this notation in place, the equation \eqref{linIE} implies
\begin{align}\label{linearizedEmbedding}
\begin{split}
X^{i(0)} &= \frac{1}{2}\left(\mathring\nabla^{A}c + \mathring\epsilon^{AB}\mathring\nabla_{B}\underline{c}\right)\mathring\nabla_{A}\tilde{X}^{i} - \frac{1}{4}\mathring\Delta c \tilde{X}^{i},\\
\mathring\nabla_{D}X^{i(0)} &= \frac{1}{2}F_{D}^{B}\mathring\nabla_{B}\tilde{X}^{i} -\frac{1}{2}\left(\mathring\nabla^{A}F_{AD}\right)\tilde{X}^{i}\\
&+\frac{1}{2}\mathring\epsilon^{AB}\mathring\nabla_{D}\mathring\nabla_{B}\underline{c}\mathring\nabla_{A}\tilde{X}^{i} - \frac{1}{2}\mathring{\epsilon}_{DB}\mathring\nabla^{B}\underline{c}\tilde{X}^{i},\\
\mathring\Delta X^{i(0)} &= -\frac{1}{2}\left(\mathring\nabla^{D}\mathring\nabla^{E}F_{DE}\right)\tilde{X}^{i} + \frac{1}{2}\mathring\epsilon^{AB}\mathring\nabla_{B}\mathring\Delta\underline{c}\mathring\nabla_{A}\tilde{X}^{i}.
\end{split}
\end{align}

Using these calculations, we proceed to expand the norm of the reference mean curvature and the connection one-form, as in the previous section.  We present the norm of the mean curvature in the following proposition:
\begin{proposition}
\begin{align}\label{H0}
\begin{split}
|H_{0,r}| &= \frac{1}{r}|H_0|^{(-1)} + \frac{1}{r^2}|H_0|^{(-2)} + \frac{1}{r^3}|H_0|^{(-3)} + O(r^{-4})\\
&=\frac{2}{r} + \frac{1}{2r^2} \na_A \na_B C^{AB} - \frac{1}{4r^3} \lt( \mathring\Delta X^{0(0)}\rt)^2 - \frac{1}{r^3} \tilde X^i \mathring\Delta X^{i(-1)} \\
&\quad - \frac{1}{2r^3} \na_A \lt( C^{AB} \na^D F_{DB}\rt) - \frac{1}{2r^3} \na_A C^{AB} \na_D \Fb^D_B \\
&\quad - \frac{1}{2r^3} C^{AB} \Fb_{AB}  + \frac{1}{16r^3} \mathring\nabla_{A}\mathring\Delta \cb\mathring\nabla^{A}\mathring\Delta \cb + \frac{1}{4r^3} \na_A C^{AB} \na_D C^D_B.
\end{split}
\end{align}
\end{proposition}
\begin{proof}
The reference mean curvature is given by $H_{0,r} = (\Delta X^0_{r}, \Delta X^{i}_{r})$, with associated expansions 
\begin{align*}
&\Delta X^0_{r} = \frac{1}{r^2}\mathring\Delta X^{0(0)} + \frac{1}{r^3}\mathring\Delta X^{0(-1)} - \frac{1}{r^3}\mathring\nabla_{B}\left(C^{AB}\mathring\nabla_{A}X^{0(0)}\right) + O(r^{-4}),\\
&\Delta X^{i}_{r} = -\frac{2}{r}\tilde{X}^{i} + \frac{1}{r^2}\mathring\Delta X^{i(0)} -\frac{1}{r^2}\left(\mathring\nabla_{B}C^{AB}\right)\mathring\nabla_{A}\tilde{X}^{i} \\
&\quad+  \frac{1}{r^3} \mathring{\Delta} X^{i(-1)} - \frac{1}{r^3}\na_A \lt( C^{AB} \na_B X^{i(0)} \rt) + \frac{1}{4r^3} \na_A \lt( C_{DE}C^{DE} \na^A \tilde X^i \rt),
\end{align*}
per the inverse metric expansion \begin{equation}\label{invMetricExpansion}
\sigma^{AB} = \frac{1}{r^2}\mathring\sigma^{AB} - \frac{1}{r^3}C^{AB} + O(r^{-4})
\end{equation}
and divergence relation \eqref{divergenceRelation}.
The norm of the mean curvature has expansion
\begin{align}
\begin{split}
|H_{0,r}| 
&= \frac{2}{r} +\frac{1}{2r^2}\left(\mathring\nabla^{D}\mathring\nabla^{E}F_{DE}\right) - \frac{1}{r^3}  \tilde X^j \mathring\Delta X^{j(-1)} - \frac{1}{4r^3} \lt( \mathring\Delta X^{0(0)}\rt)^2 \\
&\quad + \frac{1}{r^3} \tilde X^i \na_A \lt( C^{AB} \na_B X^{i(0)} \rt) + \frac{1}{2r^3} C_{AB}C^{AB}\\
&\quad + \frac{1}{4r^3} \lt( \mathring\Delta X^{i(0)} - \na_A C^{AB} \na_B \tilde X^i \rt)^2 - \frac{1}{16r^3} \lt( \na_A \na_B C^{AB} \rt)^2. \label{raw |H_{0,r}|}
\end{split}
\end{align}

Substituting for $X^{i(0)}$ per \eqref{linearizedEmbedding},
\begin{align*}
&\tilde X^i \na_A \lt( C^{AB} \na_B X^{i(0)} \rt) = \\
& -\frac{1}{2} C^{AB} C_{AB} - \frac{1}{2} \na_A \lt( C^{AB} \na^D F_{DB} \rt) - \frac{1}{2} \na_A \lt( C^{AB} \mathring\epsilon_{BD} \na^D \cb \rt)
\end{align*}
and
\begin{align*}
\frac{1}{4} \lt( \mathring\Delta X^{i(0)} - \na_A C^{AB} \na_B \tilde X^i \rt)^2 &= \frac{1}{16} \lt( \na_A \na_B C^{AB} \rt)^2 + \frac{1}{16} \mathring\nabla_{A} \mathring\Delta \cb \mathring\nabla^{A}\mathring\Delta \cb \\
&\quad + \frac{1}{4} \na_A C^{AB} \na_D C^D_B - \frac{1}{4} \na_A C^{AB} \mathring\epsilon_{BD} \na^D \mathring\Delta \cb.
\end{align*}
Adding the last term in each of these relations, we further rewrite
\begin{align*}
&- \frac{1}{4} \na_A C^{AB} \mathring\epsilon_{BD} \na^D \mathring\Delta \cb - \frac{1}{2} \na_A \lt( C^{AB} \mathring\epsilon_{BD} \na^D \cb \rt) \\
&=-\frac{1}{2} C^{AB} \Fb_{AB} - \frac{1}{2} \na_A C^{AB} \na_D \Fb_B^D.
\end{align*}
Substituting these reductions, we obtain \eqref{H0}.
\end{proof}

Turning to the reference connection one-form, we have the expansion:
\begin{proposition}
\begin{align}\label{alphaH0}
\begin{split}
&\alpha_{H_0,r}(\partial_{B}) = \frac{1}{r}\alpha_{H_0}^{(-1)}(\partial_{B}) + \frac{1}{r^2}\alpha_{H_0}^{(-2)}(\partial_{B}) + O(r^{-3})\\
&=\frac{1}{2r}\mathring\nabla_{B}\left(\mathring\Delta + 2\right)X^{0(0)} + \frac{1}{2r^2}\mathring\nabla_{B}\left(\mathring\Delta + 2\right)X^{0(-1)} \\
&-\frac{1}{2r^2}\mathring\nabla_{B}\left(\mathring\nabla_{A}\underline{F}^{AD}\mathring\nabla_{D}X^{0(0)}\right) -\frac{1}{2r^2}\left(\mathring\nabla_{B}\mathring\nabla_{E}X^{0(0)}\right)\left(\mathring\nabla_{D}F^{DE}\right)\\
& -\frac{1}{2r^2}\left(\mathring\nabla_{B}\mathring\nabla_{E}X^{0(0)}\right)\left(\mathring\epsilon^{ED}\mathring\nabla_{D}\underline{c}\right)-\frac{1}{8r^2}\mathring\nabla_{B}\left(\mathring\nabla^{D}\mathring\nabla^{E}C_{DE}\mathring\Delta X^{0(0)}\right)\\
& - \frac{1}{2r^2}\mathring\nabla_{B}\left(C^{AD}\mathring\nabla_{A}\mathring\nabla_{D}X^{0(0)}\right)- \frac{1}{2r^2}C^{D}_{B}\mathring\nabla_{D}X^{0(0)} \\
&+ \frac{1}{4r^2} \na^D X^{0(0)} \lt( \mathring\epsilon_{BE} \na_D \na^E \cb - \mathring\epsilon_{DE} \na_B \na^E \cb \rt)\\
&+ \frac{1}{2r^2}\mathring\nabla_{B}\left(\mathring\nabla^{D}X^{0(0)}\mathring{\epsilon}_{DE}\mathring\nabla^{E}\underline{c}\right)+ O(r^{-3}).
\end{split}
\end{align}
\end{proposition}
\begin{proof}
First, we expand
\[ |H_{0,r}|^{-1} = \frac{r}{2} - \frac{1}{8}\mathring\nabla^{D}\mathring\nabla^{E}C_{DE} + O(r^{-1}),\]
and
\begin{align*}
\frac{H_{0,r}}{|H_{0,r}|} &= \Bigg(\frac{1}{2r}\mathring\Delta X^{0(0)} \\
&\quad + \frac{1}{r^2}\left(\frac{1}{2}\mathring\Delta X^{0(-1)} - \frac{1}{2}\mathring\nabla_{D}\left(C^{DE}\mathring\nabla_{E}X^{0(0)}\right) - \frac{1}{8}\mathring\nabla^{D}\mathring\nabla^{E}C_{DE}\mathring\Delta X^{0(0)}\right),\\
&\quad -\tilde{X}^{i}+\frac{1}{r}\left(\frac{1}{2}\mathring\Delta X^{i(0)} - \frac{1}{2}\mathring\nabla_{D}C^{DE}\mathring\nabla_{E}\tilde{X}^{i} + \frac{1}{4}\mathring\nabla^{D}\mathring\nabla^{E}C_{DE}\tilde{X}^{i}\right)\Bigg)\\
&\quad + \Bigg(O(r^{-3}), O(r^{-2})\Bigg).
\end{align*}

Following the ideas of Chen-Wang-Yau \cite{CWY}, we note that the vector
\begin{equation}
w_r := (1+\sigma^{AB}\partial_{A}X_{r}^0\partial_{B}X_{r}^0,\sigma^{AB}\partial_{A}X_{r}^0\partial_{B}X_{r}^{i}),
\end{equation}
is normal to the embedded surface, owing to the isometric embedding equation.  Shifting by an appropriate factor, we find that
\begin{equation}
v_r := w_r - \Bigg\langle w_r , \frac{H_{0,r}}{|H_{0,r}|}\Bigg\rangle \frac{H_{0,r}}{|H_{0,r}|}
\end{equation}
is parallel to $J_{0,r}$, with length
\[ \sqrt{-\langle v_r, v_r \rangle} = 1 + O(r^{-2}).\]
Hence we have
\begin{align*}
\sqrt{-\langle v_r, v_r \rangle}\alpha_{H_0,r}(\partial_{B}) &=\Bigg\langle\nabla^{\R}_{B} v_r, \frac{H_{0,r}}{|H_{0,r}|} \Bigg\rangle \\
&= \Bigg\langle \nabla^{\R}_{B}w_r, \frac{H_{0,r}}{|H_{0,r}|} \Bigg\rangle -\nabla^{\R}_{B}\Bigg\langle w_r, \frac{H_{0,r}}{|H_{0,r}|}\Bigg\rangle.
\end{align*}

Expanding
\begin{align*}
w_r &= (1 + O(r^{-2}),\ \frac{1}{r}\mathring\nabla^{A}\tilde{X}^{i}\mathring\nabla_{A}X^{0(0)}\\
&\quad + \frac{1}{r^2}\left(\mathring\nabla^{A}X^{0(0)}\mathring\nabla_{A}X^{i(0)} + \mathring\nabla^{A}X^{0(-1)}\mathring\nabla_{A}\tilde{X}^{i} - C^{AB}\mathring\nabla_{A}X^{0(0)}\mathring\nabla_{B}\tilde{X}^{i}\right) + O(r^{-3}) ),
\end{align*}
we have
\begin{align*}
&\Bigg\langle w_r, \frac{H_{0,r}}{|H_{0,r}|} \Bigg\rangle =\\
&-\frac{1}{2r}\mathring\Delta X^{0(0)} + \frac{1}{r^2}\Big(-\frac{1}{2}\mathring\Delta X^{0(-1)} + \frac{1}{2}\mathring\nabla_{A}\left(C^{AB}\mathring\nabla_{B}X^{0(0)}\right) + \frac{1}{8}\mathring\nabla^{D}\mathring\nabla^{E}C_{DE}\mathring\Delta X^{0(0)}\\
&-\mathring\nabla^{A}X^{0(0)}\mathring\nabla_{A}X^{i(0)}\tilde{X}^{i} +\frac{1}{2}\mathring\Delta X^{i(0)}\mathring\nabla^{A}\tilde{X}^{i}\mathring\nabla_{A}X^{0(0)} - \frac{1}{2}\mathring\nabla_{A}X^{0(0)}\mathring\nabla_{D}C^{DA}\Big),
\end{align*}
and
\begin{align*}
&\Bigg\langle \nabla^{\R}_{B}w_r, \frac{H_{0,r}}{|H_{0,r}|} \Bigg\rangle \\
&=  \frac{1}{r}\mathring\nabla_{B}X^{0(0)}\\
&\quad +\frac{1}{r^2} \mathring\nabla_{B}\left(\mathring\nabla^{A}\tilde{X}^{i}\mathring\nabla_{A}X^{0(0)}\right) \left(\frac{1}{2}\mathring\Delta X^{i(0)} - \frac{1}{2}\mathring\nabla_{D}C^{DE}\mathring\nabla_{E}\tilde{X}^{i} + \frac{1}{4}\mathring\nabla^{D}\mathring\nabla^{E}C_{DE}\tilde{X}^{i}\right)\\
&\quad + \frac{1}{r^2} \lt( -C_{B}^{D}\mathring\nabla_{D}X^{0(0)} + \mathring\nabla_{B}X^{0(-1)} - \mathring\nabla_{B}\left(\mathring\nabla_{D}X^{0(0)}\mathring\nabla^{D}X^{i(0)}\right)\tilde{X}^{i}\rt) + O(r^{-3}). \label{raw alpha_{H_0,r}}
\end{align*}

Substituting for $X^{i(0)}$ via \eqref{linearizedEmbedding}, we have
\[-\mathring\nabla^{A}X^{0(0)}\mathring\nabla_{A}X^{i(0)}\tilde{X}^{i} +\frac{1}{2}\mathring\Delta X^{i(0)}\mathring\nabla^{A}\tilde{X}^{i}\mathring\nabla_{A}X^{0(0)} = \frac{1}{2}\mathring\nabla_{A}X^{0(0)}\mathring\nabla_{D}C^{DA},\]
such that
\begin{align*}
&\Bigg\langle w_r, \frac{H_{0,r}}{|H_{0,r}|} \Bigg\rangle =\\
&-\frac{1}{2r}\mathring\Delta X^{0(0)} + \frac{1}{r^2}\Big(-\frac{1}{2}\mathring\Delta X^{0(-1)} + \frac{1}{2}\mathring\nabla_{A}\left(C^{AB}\mathring\nabla_{B}X^{0(0)}\right) + \frac{1}{8}\mathring\nabla^{D}\mathring\nabla^{E}C_{DE}\mathring\Delta X^{0(0)}\Big).
\end{align*}

Per \eqref{linearizedEmbedding}, we have moreover
\begin{align*}
&\frac{1}{2}\mathring\Delta X^{i(0)} - \frac{1}{2}\mathring\nabla_{D}C^{DE}\mathring\nabla_{E}\tilde{X}^{i} + \frac{1}{4}\mathring\nabla^{D}\mathring\nabla^{E}C_{DE}\tilde{X}^{i}\\
&= -\frac{1}{2}\mathring\nabla_{D}F^{DE}\mathring\nabla_{E}\tilde{X}^{i} - \frac{1}{2}\mathring\epsilon^{ED}\mathring\nabla_{D}\underline{c}\mathring\nabla_{E}\tilde{X}^{i},
\end{align*}
and
\begin{align*}
&- \mathring\nabla_{B}\left(\mathring\nabla_{D}X^{0(0)}\mathring\nabla^{D}X^{i(0)}\right)\tilde{X}^{i}\\
&= - \mathring\nabla_{B}\left(\mathring\nabla_{D}X^{0(0)}\mathring\nabla^{D}X^{i(0)}\tilde{X}^{i}\right)+\mathring\nabla_{D}X^{0(0)}\mathring\nabla^{D}X^{i(0)}\mathring\nabla_{B}\tilde{X}^{i}\\
&= \frac{1}{2}\mathring\nabla_{B}\left(\mathring\nabla^{D}X^{0(0)}\mathring\nabla^{A}F_{AD}\right) + \frac{1}{2}\mathring\nabla_{B}\left(\mathring\nabla^{D}X^{0(0)}\mathring{\epsilon}_{DE}\mathring\nabla^{E}\underline{c}\right) \\
&\quad +\frac{1}{2}C_{B}^{D}\mathring\nabla_{D}X^{0(0)} + \frac{1}{4} \na^D X^{0(0)} \lt( \mathring\epsilon_{BE} \na_D \na^E \cb - \mathring\epsilon_{DE} \na_B \na^E \cb \rt),
\end{align*}
such that
\begin{align*}
&\Bigg\langle \nabla^{\R}_{B}w_r, \frac{H_{0,r}}{|H_{0,r}|} \Bigg\rangle = \\
& \frac{1}{r}\mathring\nabla_{B}X^{0(0)} + \frac{1}{r^2}\Big(-\frac{1}{2}\mathring\nabla_{B}\mathring\nabla_{A}X^{0(0)}\left(\mathring\nabla_{D}F^{DA} + \mathring\epsilon^{AD}\mathring\nabla_{D}\underline{c}\right)\\
&-\frac{1}{2}C_{B}^{D}\mathring\nabla_{D}X^{0(0)} + \frac{1}{4} \na^D X^{0(0)} \lt( \mathring\epsilon_{BE} \na_D \na^E \cb - \mathring\epsilon_{DE} \na_B \na^E \cb \rt) \\
&+ \mathring\nabla_{B}X^{0(-1)} +\frac{1}{2}\mathring\nabla_{B}\left(\mathring\nabla^{D}X^{0(0)}\mathring\nabla^{A}F_{AD}\right)\\
&+ \frac{1}{2}\mathring\nabla_{B}\left(\mathring\nabla^{D}X^{0(0)}\mathring{\epsilon}_{DE}\mathring\nabla^{E}\underline{c}\right)\Big) + O(r^{-3}).
\end{align*}

Collecting terms, we have the connection one-form expansion \eqref{alphaH0}.
\end{proof}

\section{The Center-of-Mass Frame}

We assume that the isometric embeddings $X_{r} = (X^0_{r}, X^i_{r})$ expanding as \eqref{embeddingExpansion} satisfy the optimal isometric embedding equation \eqref{oiee} to second order with respect to the observers
\begin{equation}\label{T0expansion}
T_{0,r} =  (1,0,0,0) + \frac{1}{r}(0,b_1,b_2,b_3) + O(r^{-2}).
\end{equation}
That is, given 
\begin{align}\label{tauExpansion}
\begin{split}
\tau_{r} := -\langle T_{0,r}, X_{r} \rangle &= \tau^{(0)} + \frac{1}{r}\tau^{(-1)} + O(r^{-2}),\\
&= (X^{0(0)} - b_{i}\tilde{X}^{i}) + \frac{1}{r}(X^{0(-1)} - b_{i}X^{i(0)}) + O(r^{-2}),
\end{split}
\end{align}
together with the data (\ref{H},\ \ref{alphaH},\ \ref{H0},\ \ref{alphaH0}), we assume that $\tau_{r}$ satisfies the equation
\[div_{\sigma}j = div_{\sigma}\left(\rho_{r}\nabla \tau_{r} -\nabla\left[\sinh^{-1}\left(\frac{\rho_{r}\Delta\tau_{r}}{|H_{0,r}||H_{r}|}\right)\right] -\alpha_{H_0,r} + \alpha_{H,r}\right) = 0\]
up to second order.  As we shall see, the embedding and observer ansatze (\ref{embeddingExpansion},\ \ref{T0expansion}) are justified assuming two simple conditions on the mass aspect function $m$ in the metric expansion \eqref{CWYexpansion}.

Using the mean curvature formulae \eqref{H} and \eqref{H0} together with the fact that $\tau_{r}$ is $O(1)$, the density \eqref{densityDef} expands as
\begin{align}\label{densityFirst}
\begin{split}
\rho_{r} &= \frac{1}{r^2}\rho^{(-2)} + \frac{1}{r^3}\rho^{(-3)} + O(r^{-4}),\\
&= \frac{2m}{r^2} + \frac{1}{r^3}\left(|H_0|^{(-3)} - |H|^{(-3)}\right) + O(r^{-4}).
 \end{split}
\end{align}

Substituting the expansions (\ref{H},\ \ref{alphaH},\ \ref{H0},\ \ref{alphaH0},\ \ref{tauExpansion},\ \ref{densityFirst}), the current \eqref{currentDef} takes the form
\begin{align}\label{jDef}
\begin{split}
&j_{A,r} = \frac{1}{r}j_{A}^{(-1)} + \frac{1}{r^2}j_{A}^{(-2)} + O(r^{-3})\\
&= \frac{1}{r}\left(\alpha_{H}^{(-1)}(\partial_{A}) - \alpha_{H_0}^{(-1)}(\partial_{A})\right) + \frac{1}{r^2}\Big(2m\mathring\nabla_{A}\tau^{(0)} - \frac{1}{2}\mathring\nabla_{A}(m\mathring\Delta\tau^{(0)})\\
& - \alpha_{H_0}^{(-2)}(\partial_{A}) + \alpha_{H}^{(-2)}(\partial_{A})\Big) + O(r^{-3})\\
&= \frac{1}{r}\left(\mathring\nabla_{A}\left(m - \frac{1}{4}\left(\mathring\nabla^{D}\mathring\nabla^{E}C_{DE}\right) - \frac{1}{2}(\mathring\Delta +2)X^{0(0)}\right)- \frac{1}{2}\mathring\nabla^{B}C_{AB}\right)\\
&+\frac{1}{r^2}\Big(2m\mathring\nabla_{A}\tau^{(0)} - \frac{1}{2}\mathring\nabla_{A}(m\mathring\Delta\tau^{(0)}) - \alpha_{H_0}^{(-2)}(\partial_{A}) + \alpha_{H}^{(-2)}(\partial_{A})\Big)\\
& + O(r^{-3}).
\end{split}
\end{align}
Observation that the optimal isometric embedding equation implies that $j_A^{(-1)}$ is a co-closed one-form:
\begin{align}\label{jFirst}
\begin{split}
j_{A}^{(-1)} &= \mathring\nabla_{A}m - \frac{1}{4}\mathring\nabla_{A}\left(\mathring\nabla^{D}\mathring\nabla^{E}C_{DE}\right) - \frac{1}{2}\mathring\nabla^{D}C_{DA} - \frac{1}{2}\mathring\nabla_{A}(\mathring\Delta +2)X^{0(0)},\\
&= -\frac{1}{2}\mathring\nabla^{D}\underline{F}_{DA}\\&= -\frac{1}{4} \mathring\epsilon_{AB} \na^B (\mathring\Delta+2)\cb.
\end{split}
\end{align}
Finally, we note that the observers $T_{0,r}$ have the form $T_{0,r} = A_{r}((1,0,0,0))$, where $A_{r}$ are Lorentz transformations expanding as
\begin{equation}\label{boostExpansion}
A_{r} = Id + \frac{1}{r}A^{(-1)} + O(r^{-2}),
\end{equation}
with
\[A^{(-1)} =
\begin{bmatrix}
 0 & b_1 & b_2 & b_3\\
 b_1 & 0 & a_{12} & a_{13}\\
 b_2 & a_{21} & 0 & a_{23}\\
 b_3 & a_{31} & a_{32} & 0
 \end{bmatrix}, \quad a_{ij} = -a_{ji}.\]

According to \eqref{jDef}, we use only the connection one-form expansions \eqref{alphaH} and \eqref{alphaH0}, along with the divergence relation \eqref{divergenceRelation}, to calculate the linear term of the optimal isometric embedding equation:
\begin{align}\label{firstOrder}
\begin{split}
&\mathring\Delta(\mathring\Delta + 2)X^{0(0)} - \mathring\Delta (2m) + \frac{1}{2}(\mathring\Delta + 2)(\mathring\nabla^{D}\mathring\nabla^{E}C_{DE}) = 0,\\
&\mathring\Delta\left[(\mathring\Delta +2)X^{0(0)} - 2m + \frac{1}{4}(\mathring\Delta +2)(\mathring\Delta +2)c\right] = 0.
\end{split}
\end{align}

We solve \eqref{firstOrder} for $X^{0(0)}$, complementing our earlier calculation of the $X^{i(0)}$ \eqref{linearizedEmbedding} per the linearized isometric embedding equations.  Integrating against the eigenfunctions $\tilde{X}^{i}$, we deduce that the equation is solvable for $X^{0(0)}$ if and only if the mass aspect satisfies
\begin{equation}\label{zeroLinMom}
\int_{S^2} m(u, x^{A})\tilde{X}^{i} = 0.
\end{equation}
That is, solvability is guaranteed by vanishing of linear momentum \eqref{LinearMomentum} at future null infinity, placing the spacetime into what is often referred to as the center-of-mass frame.  The center-of-mass frame will turn out to be essential in the calculation of center-of-mass integral.  On the other hand, the calculation for angular momentum can be modified to accommodate other linear momenta profiles at future null infinity by means of different choices of observer $T_{0,r}$.

We remark that it is possible to solve the isometric embedding equation \eqref{IE} at the next order, obtaining the embedding term $X^{i(-1)}$. 

Raising indices of the current \eqref{jDef} via \eqref{invMetricExpansion} and applying the divergence relation \eqref{divergenceRelation}, we calculate the next order term of the optimal isometric embedding equation \eqref{oiee}:
\begin{equation}\label{secondOrder}
-\frac{1}{2}\mathring\Delta\left(\mathring\Delta + 2\right)X^{0(-1)}-2b_{i}\mathring\nabla^{A}(m\mathring\nabla_{A}\tilde{X}^{i}) - b_{i}\mathring\Delta(m\tilde{X}^{i}) + S = 0,
\end{equation}
where we have used the connection one-form expansions \eqref{alphaH} and \eqref{alphaH0} in addition to that for $\tau$ \eqref{tauExpansion}.  Here $S$ is a shorthand for terms expressible in terms of the physical observables in the metric expansion \eqref{CWYexpansion} and $X^{0(0)}$.

We solve \eqref{secondOrder} for $X^{0(-1)}$, with integration against the eigenfunctions $\tilde{X}^{j}$ yielding necessary and sufficient conditions for solvability:
\begin{align*}
&\int_{S^2}\left[S\tilde{X}^{j} - 2b_{i}\mathring\nabla^{A}(m\mathring\nabla_{A}\tilde{X}^{i})\tilde{X}^{j} - b_{i}\mathring\Delta(m\tilde{X}^{i})\tilde{X}^{j}\right] = 0,\\
&\int_{S^2}\left[S\tilde{X}^{j} + 2b_{i}m\mathring\nabla_{A}\tilde{X}^{i}\mathring\nabla^{A}\tilde{X}^{j} + 2b_{i}m\tilde{X}^{i}\tilde{X}^{j}\right] = 0,\\
&\int_{S^2}S\tilde{X}^{j} +2b_{j}\int_{S^2}m = 0,
\end{align*}
where we have used the pointwise relation
\[ \mathring\nabla^{A}\tilde{X}^{i}\mathring\nabla_{A}\tilde{X}^{j} = \delta^{ij} - \tilde{X}^{i}\tilde{X}^{j}.\]

Assuming positivity of the Bondi mass \eqref{BondiMass}, such that $\int_{S^2} m > 0$, and any $S$, solvability follows from an appropriate choice of $b_{j}$, which we leave implicit.  In generating $X^{0(0)}$, $X^{0(-1)}$, and the $b_{j}$, the analysis of \eqref{firstOrder} and \eqref{secondOrder} allows us to solve for $\tau^{(0)}$ and $\tau^{(-1)}$ in \eqref{tauExpansion}.  In principle, we could also solve the isometric embedding equation \eqref{IE} to the next order, obtaining the embedding term $X^{i(-2)}$.

Determination of the $b_{i}$ is necessary in proper calculation of the terms appearing in the angular momentum calculation, though it turns out that such terms vanish after integration by parts.  On the other hand, the calculation for quasi-local center-of-mass integral relies only upon the condition \eqref{zeroLinMom} allowing solvability and application of the linearized equation \eqref{firstOrder}.

Solvability of higher-order terms appearing in the optimal isometric embedding equation \eqref{oiee} and the isometric embedding equation \eqref{IE} is accomplished in the work of Chen-Wang-Yau \cite[Theorem 3]{CWY_null}.  The authors show that, given the assumption $\int_{S^2} m > 0$, it is possible to solve \eqref{oiee} and \eqref{IE} inductively with respect to embeddings $X_{r}$ and observers $T_{0,r} = A_{r}((1,0,0,0))$ expanding according to \eqref{embeddingExpansion}, \eqref{T0expansion}, and \eqref{boostExpansion}, adding lower order terms to the Lorentz transformations $A_{r}$ as necessary.  

To summarize the section above, given our family of surfaces $\Sigma_{r}$ in a spacetime $(N, g)$ with mass aspect function $m$ satisfying
\begin{align*}
\int_{S^2} m &> 0,\\
\int_{S^2} m\tilde{X}^{i} &= 0,
\end{align*}
the pairs $(X_r, T_{0,r})$, with embeddings $X_{r}$ and observers $T_{0,r}$ expanding as \eqref{embeddingExpansion} and \eqref{T0expansion}, respectively, satisfy the optimal isometric embedding equation \eqref{oiee} and isometric embedding equation \eqref{IE} at all orders.

\section{Limit of Quasi-local Angular Momentum}
We evaluate the quasi-local angular momentum formula \eqref{qlAM} with respect to the surfaces $(\Sigma_{r}, \sigma_{r})$ in a spacetime $(N,g)$ with vanishing linear momentum at null infinity \eqref{zeroLinMom} and with positive Bondi mass, such that the pairs $(X_{r}, T_{0,r})$ described above satisfy both the isometric embedding equation \eqref{IE} and optimal isometric embedding equation \eqref{oiee} to all orders.  In particular, we can make appropriate choices of $b_{i}$ in the observer expansion of $T_{0,r}$ \eqref{T0expansion} guaranteeing solvability of the optimal isometric embedding equation at first order \eqref{firstOrder} and second order \eqref{secondOrder}.  Taking limits as $r$ approaches infinity, we recover the components of the angular momentum at future null infinity.  In doing so, we make use of the data (\ref{H},\ \ref{alphaH},\ \ref{H0},\ \ref{alphaH0}), the derived expansions (\ref{tauExpansion},\ \ref{densityFirst},\ \ref{jDef},\ \ref{boostExpansion}), the linearized optimal isometric embedding equation \eqref{firstOrder}. 

We begin by considering the case where the observer $T_0 = (1,0,0,0)$; in particular, $b_i = 0$ in its expansion \eqref{T0expansion}.  According to Definition \ref{qlCharge}, we consider the angular momentum vector fields
\begin{equation}
K_{i,j,r} = x^{i}\partial_{j} - x^{j}\partial_{i},
\end{equation}
associated with $T_0 = (1,0,0,0).$  Here $K_{i,j,r}$ denotes the restriction of the Lorentz boost to the embedded surface $X_{r}(\Sigma_{r})$.

Evaluating on $X_{r}(\Sigma_{r})$, we have
\begin{equation}\label{innerProdVanishing}
\langle K_{i,j,r}, T_0 \rangle = 0.
\end{equation}

Considering the projection on the embedded surface $X_{r}(\Sigma_{r})$, we calculate
\begin{align*}
&(K^{T}_{i,j,r})^{B} = \left(\frac{1}{r^2}\mathring\sigma^{AB} - \frac{1}{r^3}C^{AB} + O(r^{-4})\right)\cdot\\
&\Bigg\langle \left(\Big(\mathring\nabla_{A}X^{0(0)} + \frac{1}{r}\mathring\nabla_{A}X^{0(-1)} + O(r^{-2})\Big)\partial_{t} + \left(r\mathring\nabla_{A}\tilde{X}^{k} + \mathring\nabla_{A}X^{k(0)} + O(r^{-1})\right)\partial_{k}\right),\\
&\left(r\tilde{X}^{i}\partial_{j}+X^{i(0)}\partial_{j} - r\tilde{X}^{j}\partial_{i}-X^{j(0)}\partial_{i} + O(r^{-1})\right)\Bigg\rangle\\
&=  \tilde{X}^{i}\mathring\nabla^{B}\tilde{X}^{j} - \tilde{X}^{j}\mathring\nabla^{B}\tilde{X}^{i} - \frac{1}{r}C^{B}_{A}(\tilde{X}^{i}\mathring\nabla^{A}\tilde{X}^{j} - \tilde{X}^{j}\mathring\nabla^{A}\tilde{X}^{i} )\\
&+ \frac{1}{r}\mathring\nabla^{B}\left(\tilde{X}^{j}X^{i(0)}-\tilde{X}^{i}X^{j(0)}\right)+\frac{2}{r}\left(\tilde{X}^{i}\mathring\nabla^{B}X^{j(0)} - \tilde{X}^{j}\mathring\nabla^{B}X^{i(0)}\right) + O(r^{-2}),
 \end{align*}
 which simplifies to
 \begin{align}\label{projectionExpansion2}
 \begin{split}
 (K^{T}_{i,j,r})^{B} &= Y_{i,j}^{B} - \frac{1}{r}\underline{F}^{B}_{D}Y_{i,j}^{D} + \frac{1}{r}\mathring\epsilon_{AD}\mathring\nabla^{B}\mathring\nabla^{D}\underline{c}Y_{i,j}^{A}\\
 &+  \frac{1}{r}\mathring\nabla^{B}\left(\tilde{X}^{j}X^{i(0)}-\tilde{X}^{i}X^{j(0)}\right) + O(r^{-2}),
\end{split}
\end{align}
where
\begin{equation}\label{YDef}
Y_{i,j}^{A} := \tilde{X}^{i}\mathring\nabla^{A}\tilde{X}^{j} - \tilde{X}^{j}\mathring\nabla^{A}\tilde{X}^{i} = \epsilon_{ijq}\mathring\epsilon^{AB}\mathring\nabla_{B}\tilde{X}^{q},
\end{equation}
and we have substituted for the $X^{i(0)}$ in the last term via \eqref{linearizedEmbedding}.

Note that $Y_{i,j}^{A}$ satisfies the Killing equation
\begin{equation}\label{YKilling}
\mathring\nabla^{A}Y_{i,j}^{B} = -\mathring\nabla^{B}Y_{i,j}^{A},
\end{equation}
which implies the further identities
\begin{align}
\mathring\nabla_{A}Y_{i,j}^{A} &= 0, \label{angMomOne}\\
\mathring\Delta Y_{i,j}^{A} &= -Y_{i,j}^{A}. \label{angMomTwo}
\end{align}

With the above calculations, the quasi-local angular momentum \eqref{qlAM} expands as
\begin{align}
\begin{split}
J^{k}(\Sigma_{r},X_{r},T_0) = \frac{r\epsilon^{ijk}}{16\pi}&\int_{S^2} Y_{i,j}^{A}\mathring\nabla^{D}\underline{F}_{DA} + O(1),
\end{split}
\end{align}
with worrisome, possibly divergent behavior arising from the top-order term of $j_{A}$ \eqref{jFirst}.  This top-order term vanishes via integration by parts and application of the Killing equation \eqref{YKilling}.

Expanding \eqref{qlAM} to the next order, we have vanishing of integrals of exact terms in the current $j_{A}^{(-2)}$ \eqref{jDef} contracted with $Y^{A}_{i,j}$, owing to integration by parts and application of \eqref{angMomOne}. We observe 
\begin{lemma}
\begin{align*}
\int_{S^2} \lt( K^T_{i,j,r}\rt)^{(-1)A} j_A^{(-1)} =0.
\end{align*}
\end{lemma}
\begin{proof}
Recall \eqref{jFirst}, $j_A^{(-1)} = -\frac{1}{2} \na^B \Fb_{AB},$ and we have vanishing of the term
\[ \int_{S^2} \mathring\nabla^{B}\left(\tilde{X}^{j}X^{i(0)}-\tilde{X}^{i}X^{j(0)}\right)j_{B}^{(-1)} = -\frac{1}{2}\int_{S^2}\mathring\nabla^{B}\left(\tilde{X}^{j}X^{i(0)}-\tilde{X}^{i}X^{j(0)}\right)\mathring\nabla^{D}\underline{F}_{DB},\]
owing to integration by parts and
\[ \mathring\nabla^{B}\mathring\nabla^{D}\underline{F}_{BD} = 0.\]
Hence,
\[\int_{S^2} \lt( K^T_{i,j,r}\rt)^{(-1)A} j_A^{(-1)} = \frac{1}{2} \int_{S^2} \lt( Y^D_{i,j} \Fb^A_D - \mathring\epsilon_{ED}\na^A\na^D \cb Y^E_{i,j} \rt) \na^B \Fb_{AB}. \]
To show the right-hand side vanishes, we compute
 \begin{align*}
\int_{S^2} \mathring\epsilon^{AE} \na_E \tilde X^k \Fb_A^B \na^D \Fb_{DB} &= \int_{S^2} \frac{1}{2} \mathring\epsilon^{AE} \na_E \tilde X^k \Fb_A^B \mathring\epsilon_{BD} \na^D (\mathring\Delta + 2)\cb
\\ &= \int_{S^2} \frac{1}{2} \mathring\epsilon^{AE} \na_E \tilde X^k \mathring\epsilon_{BD} \na^B \Fb^D_A (\mathring\Delta + 2) \cb\\
&= \int_{S^2} - \frac{1}{4} \mathring\epsilon^{AE} \na_E \tilde X^k \na_A (\mathring\Delta + 2) \cb\ (\mathring\Delta + 2) \cb\\
&=0,
\end{align*}
where we have applied \eqref{contract_codazzi}, and
\begin{align*}
\int_{S^2} \na_D \tilde X^k \na^B \na^D \cb \na^E \Fb_{EB} &= \int_{S^2} \frac{1}{2} \na_D \tilde X^k \mathring\epsilon_{BE} \na^B\na^D \cb  \na^E (\mathring\Delta + 2) \cb\\
&= \int_{S^2} \frac{1}{2} \tilde X^k \mathring\epsilon_{BE} \na^B \cb  \na^E (\mathring\Delta + 2) \cb\\
&= \int_{S^2} -\frac{1}{2} \na^E \tilde X^k \mathring\epsilon_{BE} \na^B \cb  \mathring\Delta \cb\\
&= \int_{S^2} \frac{1}{4} \na^E \tilde X^k \mathring\epsilon_{BE} \na^B(\na_{A}\cb\na^{A}\cb)\\
&=0.
\end{align*}

Per \eqref{YDef}, $Y_{i,j}^A = \varepsilon_{ijq} \mathring\epsilon_{AB} \na^B \tilde X^q$, and the relations above imply
\begin{align}\label{vanish_YFbFb}
\int_{S^2} Y_{i,j}^A \Fb_A^B \na^D \Fb_{DB} =0,
\end{align}
and
\begin{align*}
\int_{S^2} Y_{i,j}^A \mathring\epsilon_{AD} \na^B \na^D \cb \na^E \Fb_{EB}=0.
\end{align*}
This completes the proof.
\end{proof}
\begin{remark} For future reference, we note the vanishing
\begin{align}\label{vanish_YFF}
\begin{split}
&\int_{S^2} Y_{i,j}^{A}F_{A}^{B}\mathring\nabla^{D}F_{DB} = \int_{S^2}\frac{1}{2}Y_{i,j}^{A}F_{A}^{B}\mathring\nabla_{B}\left(\mathring\Delta + 2\right)c\\
=& \int_{S^2} -\frac{1}{2}Y_{i,j}^{A}\mathring\nabla_{B}F_{A}^{B}\left(\mathring\Delta + 2\right)c = \int_{S^2}-\frac{1}{8}Y_{i,j}^{A}\mathring\nabla_{A}\left(\left(\mathring\Delta +2\right)c\right)^2 = 0,
\end{split}
\end{align}
using the properties of $Y^{A}_{i,j}$ \eqref{YKilling} and \eqref{angMomOne}.
\end{remark}

Applying these reductions and taking the limit as $r$ tends to infinity, we calculate the components of the angular momentum at future null infinity:\begin{align}\label{angmomFirst}
\begin{split}
J^k = -\frac{\epsilon^{ijk}}{8\pi}&\int_{S^2}\Bigg[- Y_{i,j}^{A}N_{A} + 2mY_{i,j}^{A}\mathring\nabla_{A}X^{0(0)}+\frac{1}{4}Y_{i,j}^{A}C_{A}^{B}\mathring\nabla^{D}C_{DB}\\
&\quad + \frac{1}{2}Y_{i,j}^{A}C_{A}^{D}\mathring\nabla_{D}X^{0(0)} +\frac{1}{2}Y_{i,j}^{A}\left(\mathring\nabla_{A}\mathring\nabla_{E}X^{0(0)}\right)\left(\mathring\nabla_{D}F^{DE}\right)\\
&\quad+\frac{1}{2}Y_{i,j}^{A}\mathring\nabla_{A}\mathring\nabla_{D}X^{0(0)}\mathring\epsilon^{DE}\mathring\nabla_{E}\underline{c} \\
&\quad-\frac{1}{4} Y^A_{i,j} \na^D X^{0(0)} \lt( \mathring\epsilon_{AE} \na_D \na^E \cb - \mathring\epsilon_{DE} \na_A \na^E \cb \rt) \Bigg].
\end{split}
\end{align}

We simplify the terms involving $X^{0(0)}$ in the following lemma:
\begin{lemma}\label{ang_X0}
\begin{align*}
&\int_{S^2} \Bigg[ - 2mY_{i,j}^{A}\mathring\nabla_{A}X^{0(0)} - \frac{1}{2}Y_{i,j}^{A}C_{A}^{D}\mathring\nabla_{D}X^{0(0)} \\ &\qquad -\frac{1}{2}Y_{i,j}^{A}\left(\mathring\nabla_{A}\mathring\nabla_{E}X^{0(0)}\right)\left(\mathring\nabla_{D}F^{DE}\right)
-\frac{1}{2}Y_{i,j}^{A}\mathring\nabla_{A}\mathring\nabla_{D}X^{0(0)}\mathring\epsilon^{DE}\mathring\nabla_{E}\underline{c} \\
&\qquad + \frac{1}{4}Y^A_{i,j} \na^D X^{0(0)} \lt( \mathring\epsilon_{AE} \na_D \na^E \cb - \mathring\epsilon_{DE} \na_A \na^E \cb \rt)\Bigg]\\
&= \int_{S^2} - cY_{i,j}^{A}\mathring\nabla_{A}m  
\end{align*}
\end{lemma}
\begin{proof}
We will use the following identities obtained by integration by parts and (\ref{YKilling}, \ref{angMomOne}, \ref{angMomTwo}): For any functions $f$ on $S^2,$
\begin{align}\label{by_parts_1}
\int_{S^2} Y^A_{i,j} \na_A f (\mathring\Delta+2) f =0;
\end{align}
for any functions $f$ and $g$ on $S^2,$ 
\begin{align}\label{by_parts_2}
\int_{S^2} Y^A_{i,j}  (\mathring\Delta+2) f \na_A g= \int_{S^2} Y^A_{i,j} f \na_A (\mathring\Delta+2)g.
\end{align}

With repeated integration by parts and application of the properties of $Y^{A}_{i,j}$ (\ref{YKilling},\ \ref{angMomOne},\ \ref{angMomTwo}), we calculate
\begin{align*}
&\int_{S^2}-\frac{1}{2}Y_{i,j}^{A}\left(\mathring\nabla_{A}\mathring\nabla_{E}X^{0(0)}\right)\left(\mathring\nabla_{D}F^{DE}\right)\\
=&\int_{S^2}\left[\frac{1}{4}\left(\mathring\nabla^{E}Y_{i,j}^{A}\right)\mathring\nabla_{A}X^{0(0)}\mathring\nabla_{E}\left(\mathring\Delta + 2\right)c + \frac{1}{2}Y_{i,j}^{A}\mathring\nabla_{A}X^{0(0)}\left(\mathring\nabla^{D}\mathring\nabla^{E}F_{DE}\right)\right]\\
=&\int_{S^2}\left[-\frac{1}{4}\mathring\Delta Y_{i,j}^{A}\mathring\nabla_{A}X^{0(0)}\left(\mathring\Delta + 2\right)c + \frac{1}{2}Y_{i,j}^{A}\mathring\nabla_{A}X^{0(0)}\left(\mathring\nabla^{D}\mathring\nabla^{E}C_{DE}\right)\right]\\
=&\int_{S^2}\left[-\frac{1}{2}Y_{i,j}^{A}X^{0(0)}\mathring\nabla^{D}F_{DA} + \frac{1}{2}Y_{i,j}^{A}\mathring\nabla_{A}X^{0(0)}\left(\mathring\nabla^{D}\mathring\nabla^{E}C_{DE}\right)\right]\\
=&\int_{S^2}\left[\frac{1}{2}Y_{i,j}^{A}\mathring\nabla_{D}X^{0(0)}F^{D}_{A} + \frac{1}{2}Y_{i,j}^{A}\mathring\nabla_{A}X^{0(0)}\left(\mathring\nabla^{D}\mathring\nabla^{E}C_{DE}\right)\right]\\
\end{align*}
and
\begin{align*}
&\frac{1}{4} \int_{S^2} Y^A_{i,j} \na^D X^{0(0)} \lt( \mathring\epsilon_{AE} \na_D \na^E \cb - \mathring\epsilon_{DE} \na_A \na^E \cb \rt)\\
=&\frac{1}{4} \int_{S^2} -Y^A_{i,j} \mathring\Delta X^{0(0)} \mathring\epsilon_{AE} \na^E \cb + \na^A Y^D_{i,j} \na_D X^{0(0)} \mathring\epsilon_{AE}\na^E \cb + Y^A_{i,j} \na^D \na_A X^{0(0)} \mathring\epsilon_{DE} \na^E \cb\\
=& -\frac{1}{4} \int_{S^2} Y^A_{i,j} \mathring\Delta X^{0(0)} \mathring\epsilon_{AE} \na^E \cb.  
\end{align*}
Hence, the original integral reduces to
\begin{align*}
\int_{S^2} \Bigg[&-2mY_{i,j}^{A}\mathring\nabla_{A}X^{0(0)} + \frac{1}{2}Y_{i,j}^{A}\mathring\nabla_{A}X^{0(0)}\left(\mathring\nabla^{D}\mathring\nabla^{E}C_{DE}\right) \\
& -\frac{1}{2} Y^A_{i,j} \underline{F}_A^D
 \mathring{\nabla}_D X^{0(0)} -\frac{1}{2}Y_{i,j}^{A}\mathring\nabla_{A}\mathring\nabla_{D}X^{0(0)}\mathring\epsilon^{DE}\mathring\nabla_{E}\underline{c} -\frac{1}{4} Y^A_{i,j} \mathring\Delta X^{0(0)} \mathring\epsilon_{AE} \na^E \cb \Bigg]. 
\end{align*}
Substituting the linearized equation \eqref{firstOrder} into the first line, we calculate
\begin{align*}
&\int_{S^2}\left[-2mY_{i,j}^{A}\mathring\nabla_{A}X^{0(0)} + \frac{1}{2}Y_{i,j}^{A}\mathring\nabla_{A}X^{0(0)}\left(\mathring\nabla^{D}\mathring\nabla^{E}C_{DE}\right)\right]\\
&= \int_{S^2} \left[-(\mathring\Delta +2)X^{0(0)}Y_{i,j}^{A}\mathring\nabla_{A}X^{0(0)} - \frac{1}{2}(\mathring\Delta +2)cY_{i,j}^{A}\mathring\nabla_{A}X^{0(0)}\right]\\
&= \int_{S^2} - c Y^A_{i,j} \na_A m,
\end{align*}
where we used \eqref{by_parts_1}, \eqref{by_parts_2} and linearized equation \eqref{firstOrder} again in the last equality.

For the second line, using the properties \eqref{YKilling} and \eqref{angMomTwo}, we note
\begin{align*}
&\int_{S^2} Y_{i,j}^{A}\underline{F}_{A}^{D}\mathring\nabla_{D}X^{0(0)} \\
&=\int_{S^2} -\frac{1}{2}Y_{i,j}^{A}X^{0(0)}\mathring\epsilon_{AD}\mathring\nabla^{D}\left(\mathring\Delta + 2\right)\underline{c}\\
&=\int_{S^2}\left[-Y_{i,j}^{A}X^{0(0)}\mathring\epsilon_{AD}\mathring\nabla^{D}\underline{c} - \frac{1}{2}Y_{i,j}^{A}X^{0(0)}\mathring\epsilon_{AD}\mathring\nabla^{D}\mathring\Delta\underline{c}\right]\\
&=\int_{S^2}\left[-\frac{1}{2}Y_{i,j}^{A}X^{0(0)}\mathring\epsilon_{AD}\mathring\nabla^{D}\underline{c} - \frac{1}{2}\mathring\Delta\left(Y_{i,j}^{A}X^{0(0)}\right)\mathring\epsilon_{AD}\mathring\nabla^{D}\underline{c}\right]\\
&=\int_{S^2}\left[-\frac{1}{2}Y_{i,j}^{A}\mathring\Delta X^{0(0)}\mathring\epsilon_{AD}\mathring\nabla^{D}\underline{c} - \mathring\nabla^{B}Y_{i,j}^{A}\mathring\nabla_{B}X^{0(0)}\mathring\epsilon_{AD}\mathring\nabla^{D}\underline{c}\right]\\
&=\int_{S^2}\left[-\frac{1}{2}Y_{i,j}^{A}\mathring\Delta X^{0(0)}\mathring\epsilon_{AD}\mathring\nabla^{D}\underline{c} - Y_{i,j}^{B}\mathring\nabla^{A}\mathring\nabla_{B}X^{0(0)}\mathring\epsilon_{AD}\mathring\nabla^{D}\underline{c}\right].
\end{align*}
Hence, the second line vanishes. Putting the two reductions together, we obtain Lemma \ref{ang_X0}.
\end{proof}

Applying these reductions, we have a simplified formula for the angular momentum at future null infinity:
\begin{theorem}\label{main_angmom}
Suppose $(N,g)$ is a spacetime with Bondi-Sachs expansion \eqref{CWYexpansion}.  Further assuming that $(N,g)$ has vanishing linear momentum at null infinity \eqref{zeroLinMom} and positive Bondi mass, the angular momentum at null infinity has the components
\begin{align}\label{angmomFinal}
\begin{split}
J^{k} = \frac{\epsilon^{ijk}}{8\pi}\int_{S^2}& Y_{i,j}^{A} \lt( N_{A} -\frac{1}{4}C_{AB}\mathring\nabla_{D}C^{DB} - c\mathring\nabla_{A}m   \rt) .
\end{split}
\end{align} 
In particular, if $C_{AB}$ is closed, then
\[J^k =\frac{\epsilon^{ijk}}{8\pi} \int_{S^2} Y_{i,j}^{A} \left( N_{A} - c \mathring\nabla_{A}m \right),\]
thanks to \eqref{vanish_YFF}.
In addition, if $C_{AB}$ is closed and $m$ is angular independent, then
\[J^{k} = \frac{\epsilon^{ijk}}{8\pi} \int_{S^2} Y_{i,j}^{A}N_{A}.\]
\end{theorem} 

\begin{proof}
We have calculated \eqref{angmomFinal} with respect to an observer $T_0 = (1,0,0,0)$, assuming that \eqref{firstOrder} and \eqref{secondOrder} are solvable.  More generally, we have solvability of the two for appropriate choice of $b_{i}$ in the observer expansion \eqref{T0expansion}, assuming vanishing of linear momentum at null infinity \eqref{zeroLinMom} and positivity of the Bondi mass.  Following Definition \ref{qlCharge}, we write $T_{0,r} = A_{r}((1,0,0,0))$, with $SO(3,1)$ transformations $A_{r}$, and measure the quasi-local angular momentum with respect to the Lorentz rotation $A_{r}(K_{i,j,r})$.

Owing to preservation under $SO(3,1)$ transformations, vanishing of the inner product \eqref{innerProdVanishing} still holds. Let $\varepsilon^{(ij)}$ denote the $3 \times 3$ skew-symmetric matrix with $ij$ entry 1, $ji$ entry $-1$ and other entries $0$. Using the expansion of the Lorentz transformations $A_{r}$ \eqref{boostExpansion}, we have
\[ A_r(K_{i,j,r}) = x^i\partial_{j} - x^{j}\partial_{i} + \frac{1}{r}\left((x^ib^j - x^jb^i)\partial_{t} + \sum_{k,l=1}^3 a_{kl} \varepsilon^{(ij)}_{lp} x^p \pl_k \right)+ O(r^{-2}),\]
from which application of the projection formula \eqref{projectionFormula} gives the projection expansion 
 \begin{align*}
&\lt( A_r (K^{T}_{i,j,r}) \rt)^{B} = Y_{i,j}^{B} - \frac{1}{r}\underline{F}^{B}_{D}Y_{i,j}^{D} + \frac{1}{r}\mathring\epsilon_{AD}\mathring\nabla^{B}\mathring\nabla^{D}\underline{c}Y_{i,j}^{A}\\
 &+  \frac{1}{r}\mathring\nabla^{B}(\tilde{X}^{j}X^{i(0)}-\tilde{X}^{i}X^{j(0)}) + \frac{1}{r}\left(\sum_{k,l=1}^3 a_{kl} \varepsilon^{(ij)}_{lp} \tilde X^p \na^B \tilde X^k\right) + O(r^{-2}),
\end{align*}
differing from the earlier expansion \eqref{projectionExpansion2} by an exact term.  The contribution to the quasi-local formula amounts to
\[ \frac{1}{16\pi} \sum_{k,l} a_{kl} \varepsilon^{(ij)}_{pl} \int_{S^2} \tilde X^p \na^B \tilde X^k j_{B}^{(-1)} = 0,\]
by \eqref{jFirst} and Lemma \ref{xixj-co-closed}.

There are also changes in $j_{A}^{(-2)}$, via the definition of $j_{A}$ \eqref{jDef} and the expansion of $\tau$ \eqref{tauExpansion}.  The modifications to the quasi-local formula vanish, as
\begin{align*}
&\frac{1}{8\pi}\int_{S^2}\left[2mY_{i,j}^{A}\mathring\nabla_{A}b_{n}\tilde{X}^{n} + Y_{i,j}^{A}\mathring\nabla_{A}\left(m b_{n}\tilde{X}^{n}\right)\right]\\
&=\frac{b_{n}}{8\pi}\int_{S^2} 2m \epsilon_{ijq}\lt( \mathring\epsilon^{AB}\mathring\nabla_{B}\tilde{X}^{q}\mathring\nabla_{A}\tilde{X}^{n} \rt) \\
&=\frac{b_n}{8\pi} \int_{S^2} 2m \epsilon_{ijq} \epsilon^{nq}_{\;\;\;\;l} \tilde X^l =0,
\end{align*}
where we have used the definition of $Y^{A}_{i,j}$ \eqref{YDef}, its divergence-free property \eqref{angMomOne}, and the assumption of vanishing linear momentum \eqref{zeroLinMom}.

With the vanishing of these terms and taking limits, we see that the earlier formula \eqref{angmomFinal} is preserved.
\end{proof}

\section{Limit of Quasi-local Center-of-Mass Integral}

We evaluate the quasi-local center-of-mass formula \eqref{qlCOM} with respect to the surfaces $(\Sigma_{r}, \sigma_{r})$ in a spacetime $(N,g)$ with vanishing linear momentum at null infinity \eqref{zeroLinMom} and with positive Bondi mass, such that the pairs $(X_{r}, T_{0,r})$ described above satisfy both the isometric embedding equation \eqref{IE} and optimal isometric embedding equation \eqref{oiee} to all orders.  Taking limits as $r$ approaches infinity, we recover the components of the center-of-mass integral at future null infinity.  In doing so, we make use of the data (\ref{H},\ \ref{alphaH},\ \ref{H0},\ \ref{alphaH0}), the derived expansions (\ref{tauExpansion},\ \ref{densityFirst},\ \ref{jDef},\ \ref{boostExpansion}), the linearized optimal isometric embedding equation \eqref{firstOrder}.

For simplicity, we begin by considering the case where the observers $T_{0,r} = (1,0,0,0)$, such that $b_i = 0$ in its expansion \eqref{T0expansion}.  According to Definition \ref{qlCharge}, we consider the boosts
\begin{equation}
K_{i,r} := x^{i}\partial_{t} + t\partial_{i} 
\end{equation}
associated with $T_0 = (1,0,0,0).$  Here $K_{i,r}$ denotes the restriction of the Lorentz boost to the embedded surface $X_{r}(\Sigma_{r})$.

On the embedded surfaces $X_{r}(\Sigma_{r})$, the $K_{i,r}$ satisfy
\begin{equation}\label{innerProdExpansion}
\langle K_{i,r}, T_0 \rangle = -r\tilde{X}^{i} - X^{i(0)} + O(r^{-1}),
\end{equation}
\begin{align}\label{projectionExpansion}
\begin{split}
(K^{T}_{i,r})^{B} &= \left(\frac{1}{r^2}\mathring\sigma^{AB} - \frac{1}{r^3}C^{AB} + O(r^{-4})\right)\\
&\Bigg\langle \left(r\tilde{X}^{i} +X^{i(0)}\right)\partial_{t} + X^{0(0)}\partial_{i} + O(r^{-1}),\\
&\mathring\nabla_{A}X^{0(0)}\partial_{t} + \left(r\mathring\nabla_{A}\tilde{X}^{j} + \mathring\nabla_{A}X^{j(0)}\right)\partial_{j} + O(r^{-1})\Bigg\rangle\\
&=\frac{1}{r}\left(-\mathring\nabla^{B}\left(\tilde{X}^{i}X^{0(0)}\right) + 2\mathring\nabla^{B}\tilde{X}^{i}X^{0(0)}\right) + O(r^{-2}),
 \end{split}
 \end{align}
 where in the second expression we apply the projection formula \eqref{projectionFormula}.

Expanding the center-of-mass formula \eqref{qlCOM}, we find
 \begin{equation}
 C^i(\Sigma_{r},X_{r},T_0)  = \frac{r}{4\pi e}\int_{S^2}m\tilde{X}^{i} + O(1),
 \end{equation}
 with the seemingly divergent top-order term annihilated by our assumption of vanishing linear momentum at null infinity \eqref{zeroLinMom}.  

Expanding \eqref{qlCOM} to the next order and taking the limit as $r$ approaches infinity, the center-of-mass integral at future null infinity is given by components
\begin{align}\label{comFirst}
\begin{split}
C^i = \frac{1}{8\pi e} \int_{S^2} &\Big[\tilde X^i \lt( |H_0|^{(-3)} - |H|^{(-3)} \rt)\\
 &+ 2mX^{i(0)}- 2 \na^A \tilde X^i X^{0(0)} j_A^{(-1)}\Big],
 \end{split}
\end{align}
where we have applied the linearized optimal isometric embedding equation \eqref{firstOrder}, amounting to a divergence-free condition on $j_{A}^{(-1)}$, to integrate away its contraction with the first term in the expansion \eqref{projectionExpansion}.

As mentioned in the previous section, owing to the form of the components of \eqref{comFirst}, calculation of the center-of-mass does not rely upon application of the second order term \eqref{secondOrder} in the optimal isometric embedding equation.  

We begin our simplification of \eqref{comFirst} by rewriting terms in the integral of the reference mean curvature norm \eqref{H0}:
\begin{lemma}\label{X(-1)}
\begin{align*}
&\int_{S^2} \tilde X^i \lt( - \tilde X^j \mathring{\Delta} X^{j(-1)} - \frac{1}{4} \lt( \mathring{\Delta} X^{0(0)}\rt)^2 \rt) \\
&= \int_{S^2} \tilde X^i \Bigg( \frac{1}{4}(C_{DE}C^{DE}) + \frac{1}{4}(\Fb_{DE}\Fb^{DE}) -\frac{1}{4}\mathring\nabla_{A}\mathring\nabla_{B}\cb\mathring\nabla^{A}\mathring\nabla^{B}\cb - \frac{1}{2} \na^B F_{BD} \mathring\epsilon^{DA} \na_A \cb\Bigg)\\
&\quad +\int_{S^2}\tilde{X}^{i}\Bigg(- \frac{1}{4} \mathring\nabla_{A}\cb\mathring\nabla^{A}\cb -\frac{1}{16} \lt( \na^A \na^B C_{AB}\rt)^2 -m^2 + \frac{1}{2} m  (\mathring{\Delta} + 2) c + \frac{1}{2} m \na^A \na^B C_{AB} \Bigg).
\end{align*}
\end{lemma}
\begin{proof}
Suppose $X^{j(-1)} = \alpha^A \na_A \tilde X^j + \beta \tilde X^j$.  With this notation and the surface metric expansion \eqref{metricExpansion}, the second order term of the isometric embedding equation \eqref{IE} takes the form
\begin{align*}
&\na_A \alpha_B + \na_B \alpha_A + 2 \beta \mathring{\sigma}_{AB} \\
&= \frac{1}{4}(C_{DE}C^{DE}) \mathring{\sigma}_{AB} - \na_A X^{j(0)} \na_B X^{j(0)} + \na_A X^{0(0)} \na_B X^{0(0)} =: \delta\sigma_{AB}.
\end{align*}
We compute
\[ - \tilde X^j \mathring{\Delta} X^{j(-1)} = \mathring\sigma^{AB} \delta\sigma_{AB} - (\mathring{\Delta}+2)\beta. \]

Substituting for $X^{i(0)}$ \eqref{linearizedEmbedding}, we get
\begin{align*}
\mathring\sigma^{AB} \delta\sigma_{AB} &= \frac{1}{2}(C_{AB}C^{AB}) -\frac{1}{4}(F_{AB}F^{AB}) -\frac{1}{4} \na^A F_{AD} \na^B F_B^D - \frac{1}{2} (F_{AB} \Fb^{AB}) \\
&- \frac{1}{4} \mathring\nabla_{A}\mathring\nabla_{B}\cb\mathring\nabla^{A}\mathring\nabla^{B}\cb
- \frac{1}{2} \na^B F_{BD} \mathring\epsilon^{DA} \na_A \cb - \frac{1}{4}\mathring\nabla_{A}\cb\mathring\nabla^{A}\cb + |\na X^{0(0)}|^2. 
\end{align*}
Noting that
\[ \int_{S^2} \tilde X^i \lt( |\na X^{0(0)}|^2 - \frac{1}{4} (\mathring\Delta X^{0(0)})^2 \rt) = \int_{S^2} -\frac{1}{4} \tilde X^i \lt( (\mathring\Delta + 2) X^{0(0)}\rt)^2, \]
we substitute linearized optimal isometric embedding equation \eqref{firstOrder} for $(\mathring\Delta + 2)X^{0(0)}$ to obtain
\begin{align*}
&\int_{S^2} \tilde X^i \lt( - \tilde X^j \mathring{\Delta} X^{j(-1)} - \frac{1}{4} \lt( \mathring{\Delta} X^{0(0)}\rt)^2 \rt) \\
&= \int_{S^2} \tilde X^i \Bigg( \frac{1}{4}(C_{DE}C^{DE}) + \frac{1}{4}(\Fb_{DE}\Fb^{DE}) -\frac{1}{4}\mathring\nabla_{A}\mathring\nabla_{B}\cb\mathring\nabla^{A}\mathring\nabla^{B}\cb - \frac{1}{2} \na^B F_{BD} \mathring\epsilon^{DA} \na_A \cb\Bigg)\\
&\quad +\int_{S^2}\tilde{X}^{i}\Bigg(- \frac{1}{4} \mathring\nabla_{A}\cb\mathring\nabla^{A}\cb -\frac{1}{16} \lt( \na^A \na^B C_{AB}\rt)^2 -m^2 + \frac{1}{2} m  (\mathring{\Delta} + 2) c + \frac{1}{2} m \na^A \na^B C_{AB} \Bigg)\\
&\quad - \int_{S^2} \tilde X^i \lt( \frac{1}{4} \na^A F_{AD}\na^B F_B^D + \frac{1}{16} \lt( (\mathring\Delta +2)c \rt)^2 + \frac{1}{8} \na^A\na^B C_{AB} (\mathring\Delta+2)c \rt).
\end{align*}
Recall $\na^B F_B^D = \frac{1}{2} \na^D(\mathring\Delta + 2)c$ and the last integral vanishes via integration by parts:
\begin{align*}
&\int_{S^2} \tilde X^i \lt( \frac{1}{4} \na^A F_{AD}\na^B F_B^D + \frac{1}{16} \lt( (\mathring\Delta +2)c \rt)^2 + \frac{1}{8} \na^A\na^B C_{AB} (\mathring\Delta+2)c \rt)\\
&= \int_{S^2} -\frac{1}{8} \na^D \tilde X^i \na^A F_{AD} (\mathring\Delta+2)c + \frac{1}{16} \tilde X^i \lt( (\mathring\Delta+2)c\rt)^2\\
&= \int_{S^2} - \frac{1}{32} \na^D \tilde X^i \na_D \lt( (\mathring\Delta+2)c \rt)^2 + \frac{1}{16} \tilde X^i \lt( (\mathring\Delta+2)c \rt)^2\\
&=0.
\end{align*}
This completes the proof.

\end{proof}

Applying the mean curvature expansions \eqref{H} and \eqref{H0} together with Lemma \ref{X(-1)}, we find
\begin{align*}
&\int_{S^2} \tilde X^i \lt( |H_0|^{(-3)} - |H|^{(-3)}\rt) \\
&= \int_{S^2} \left[\na^A \tilde X^i N_A + \frac{1}{2} \tilde X^i m (\mathring\Delta + 2) c  + \frac{1}{8}\tilde{X}^{i}(\mathring\Delta +2)(C_{DE}C^{DE})\right]\\
&+ \int_{S^2} \tilde X^i \lt( \frac{1}{4}(\Fb_{DE}\Fb^{DE}) -\frac{1}{4} (\mathring\nabla_{A}\mathring\nabla_{B}\cb\mathring\nabla^{A}\mathring\nabla^{B}\cb) - \frac{1}{2} \na^B F_{BD} \mathring\epsilon^{DA} \na_A \cb - \frac{1}{4}\mathring\nabla^{A}\cb\mathring\nabla_{A}\cb \rt)\\
&+ \int_{S^2} \tilde X^i \lt( -\frac{1}{2} \na_A C^{AB} \na_D \Fb_B^D + \frac{1}{2} \na_A \lt( C^{AB} \na^D \Fb_{DB} \rt) - \frac{1}{2} C^{AB} \Fb_{AB} + \frac{1}{16} \mathring\nabla_{A}\mathring\Delta \cb \mathring\nabla^{A}\mathring\Delta \cb \rt).
\end{align*}
Note that \begin{align*}
\int_{S^2} -\frac{1}{4} \tilde X^i (\mathring\nabla_{A}\mathring\nabla_{B}\cb\mathring\nabla^{A}\mathring\nabla^{B}\cb)
&= \int_{S^2} \left[\frac{1}{4} \na_A \tilde X^i \na_B \cb \na^A \na^B \cb + \frac{1}{4} \tilde X^i \na_B \cb \mathring\Delta \na^B \cb\right]\\
&= \int_{S^2} \left[\frac{1}{2} \tilde X^i \na_{A}\cb\na^{A}\cb + \frac{1}{4} \tilde X^i \na_B \cb \na^B \mathring\Delta \cb\right]
\end{align*}
and hence
\begin{align*}
&\int_{S^2} \tilde X^i \lt[ -\frac{1}{4} \na_A \na_B \cb \na^A \na^B \cb - \frac{1}{4} \na^A \cb \na_A \cb + \frac{1}{16} \na_A \mathring\Delta \cb \na^A \mathring\Delta \cb \rt]\\
&= \int_{S^2} \frac{1}{16} \tilde X^i \na_A (\mathring\Delta + 2)\cb \na^A (\mathring\Delta + 2)\cb\\
&= \int_{S^2} \frac{1}{4} \tilde X^i \na_A \Fb^{AB} \na^D \Fb_{BD}. 
\end{align*}
Moreover, by \eqref{contract_codazzi}, we have
\begin{align*}
\int_{S^2} \tilde X^i \lt( \na^B F_{BD} \mathring\epsilon^{DA} \na_A \cb + F^{AB} \Fb_{AB} \rt) &= \int_{S^2} - \na^B \tilde X^i F_{BD} \mathring\epsilon^{DA} \na_A \cb\\
&= \int_{S^2} -\frac{1}{2} \cb\na^B \tilde X^i \mathring\epsilon_{BA} \na^A (\mathring\Delta + 2) c \\
&= \int_{S^2} -\frac{1}{2} \tilde X^i \na^B (\mathring\Delta + 2) c\ \mathring\epsilon_{BA} \na^A \cb\\
&= \int_{S^2} - \tilde X^i \na^B F_{BD} \mathring\epsilon^{DA} \na_A \cb
\end{align*}
and hence
\begin{align*}
\int_{S^2} -\frac{1}{2} \tilde X^i \na^B F_{BD} \mathring\epsilon^{DA}\na_A \cb = \int_{S^2} \frac{1}{4} \tilde X^i F^{AB} \Fb_{AB}.
\end{align*}
These two reductions imply
\begin{align*}
&\int_{S^2} \tilde X^i \lt( |H_0|^{(-3)} - |H|^{(-3)} \rt)\\
&= \int_{S^2} \lt[ \na^A \tilde X^i N_A + \frac{1}{2} \tilde X^i m (\mathring\Delta + 2)c \rt]\\
&\quad+ \int_{S^2} \tilde X^i \lt( \frac{1}{4} \na_A \Fb^{AB} \na^D \Fb_{BD} -\frac{1}{2} \na_A C^{AB} \na_D \Fb_B^D + \frac{1}{2} \na_A \lt( C^{AB} \na^D \Fb_{DB} \rt) - \frac{1}{4} C^{AB} \Fb_{AB} \rt). 
\end{align*}

The next lemma evaluates the second part of the center-of-mass integral \eqref{comFirst}. 
\begin{lemma}\label{last_two_terms}
\begin{align*}
\int_{S^2} &\left[2mX^{i(0)}- 2 \na^A \tilde X^i X^{0(0)} j_A^{(-1)}\right] \\
= \int_{S^2} &\Big[\frac{1}{2} \tilde X^i \na_D F^{AD} \na^B \Fb_{AB} - \na_A \tilde X^i c \na^A m + 2 \tilde X^i c m \\
&- \frac{1}{2} \tilde X^i (\mathring\Delta c) m + 2 \mathring\epsilon^{AB} \na_A \tilde X^i (\na_B \cb) m\Big].
\end{align*}
\end{lemma}
\begin{proof}
Recall \eqref{jFirst} \[ j_A^{(-1)} = -\frac{1}{2} \na^B \Fb_{AB} = -\frac{1}{4} \mathring\epsilon_{AB}\na^B (\mathring\Delta + 2)\cb \] and note that for any functions $f$ and $g$
\begin{align}\label{by_parts_epsilon}
\int_{S^2} \na^A \tilde X^i f \mathring\epsilon_{AB} \na^B (\mathring\Delta+2)g = \int_{S^2} \na^A \tilde X^i (\mathring\Delta+2)f \mathring\epsilon_{AB} \na^B g.
\end{align} The assertion follows from the linearized optimal isometric embedding equation \eqref{firstOrder} and the expressions for $X^{i(0)}$ \eqref{linearizedEmbedding}.
\end{proof}

In summary, the center-of-mass integral is equal to (up to the factor $\frac{1}{8\pi e}$)
\begin{align}
\begin{split}
&\int_{S^2} \na^A \tilde X^i \lt( N_A - c \na_A m + 2 \mathring\epsilon_{AB} (\na^B \cb)m \rt) + 3 \tilde X^i cm\\
&+ \int_{S^2} \tilde X^i \lt( -\frac{1}{4} \na_A \Fb^{AB} \na^D \Fb_{BD} + \frac{1}{2} \na_A \lt( C^{AB} \na^D \Fb_{DB} \rt) - \frac{1}{4} C^{AB} \Fb_{AB} \rt) \label{pre_com}
\end{split}
\end{align}
The last integral will be simplified by the next lemma. 
\begin{lemma}
Per Proposition \ref{interchange}, we have
\begin{align}
\int_{S^2} \tilde X^i \na_A \lt( F^{AB} \na^D F_{DB}\rt) &= \int_{S^2} \frac{1}{2} \tilde X^i F_{AB} F^{AB}, \label{integral_identity_FF}\\
\int_{S^2} \tilde X^i \na_A \lt( \Fb^{AB} \na^D \Fb_{DB}\rt) &= \int_{S^2} \frac{1}{2} \tilde X^i \Fb_{AB} \Fb^{AB},\label{integral_identity_FbFb}\\
\int_{S^2} \tilde X^i \na_A \lt( F^{AB} \na^D \Fb_{DB}\rt) &= \int_{S^2} \tilde X^i \na_A \lt( \Fb^{AB} \na^D F_{DB}\rt) \notag \\
&= - \int_{S^2} \tilde X^i \na_A F^{AB} \na^D \Fb_{DB}. \label{integral_identity_FFb}
\end{align}
\end{lemma}
\begin{proof} Using \eqref{contract_codazzi}, we have
\begin{align*}
\int_{S^2} \tilde X^i \na_A \lt( \Fb^{AB} \na^D \Fb_{DB} \rt) &= \int_{S^2} - \na_A \tilde X^i \Fb^{AB} \na^D \Fb_{DB}\\
&= \int_{S^2} -\frac{1}{2} \na_A \tilde X^i \Fb^{AB} \mathring\epsilon_{BD} \na^D (\mathring\Delta + 2)\cb\\
&= \int_{S^2} \frac{1}{2} \na_A \tilde X^i \mathring\epsilon_{BD} \na^D \Fb^{BA} (\mathring\Delta + 2)\cb\\
&= \int_{S^2} \frac{1}{4} \na_A \tilde X^i \na^A (\mathring\Delta + 2)\cb (\mathring\Delta + 2)\cb\\
&= \int_{S^2} \frac{1}{4} \tilde X^i \lt( (\mathring\Delta + 2)\cb \rt)^2,
\end{align*}
and \begin{align*}
&\int_{S^2}  \tilde X^i \Fb^{AB}\Fb_{AB} \\
&= \int_{S^2}  \tilde X^i \Fb^{AB} \mathring\epsilon_{AD} \na_B \na^D \cb\\
&= \int_{S^2} \left[ -\na^D \tilde X^i \mathring\epsilon_{AD} \Fb^{AB} \na_B \cb - \frac{1}{2} \tilde X^i \na^B (\mathring\Delta + 2) \cb \na_B \cb\right]\\
&= \int_{S^2} \frac{1}{2} \left[  \na^D \tilde X^i \na_D (\mathring\Delta + 2) \cb \cdot \cb  + \na^B \tilde X^i (\mathring\Delta + 2) \cb \na_B \cb + \tilde X^i (\mathring\Delta + 2) \cb \mathring\Delta \cb\right]\\
&= \int_{S^2} \frac{1}{2} \tilde X^i \lt( (\mathring\Delta + 2) \cb \rt)^2.
\end{align*}
This proves \eqref{integral_identity_FbFb}. Identity \eqref{integral_identity_FF} is proved similarly using \eqref{contract_codazzi}.

For \eqref{integral_identity_FFb}, we have
\begin{align*}
\int_{S^2} \tilde X^i \na_A \lt( F^{AB} \na^D \Fb_{DB}\rt) &= \int_{S^2} -\frac{1}{2} \na_A \tilde X^i F^{AB} \mathring\epsilon_{BD} \na^D (\mathring\Delta + 2) \cb \\
&= \int_{S^2} -\frac{1}{4} \na_A \tilde X^i \mathring\epsilon^{AB} \na_B (\mathring\Delta + 2) c (\mathring\Delta + 2)\cb\\
&= \int_{S^2} -\tilde X^i \na^D F_{BD} \na_A \Fb^{AB},
\end{align*}
using \eqref{contract_codazzi} and 
\begin{align*}
\int_{S^2} \tilde X^i \na_A \lt( \Fb^{AB} \na^D F_{DB}\rt)
&= \int_{S^2} - \frac{1}{2} \na_A \tilde X^i \Fb^{AB} \na_B (\mathring\Delta + 2)c \\
&= \int_{S^2} \frac{1}{2} \na_A \tilde X^i \na_B \Fb^{AB} (\mathring\Delta + 2 )c\\
&= \int_{S^2} -\tilde X^i \na^D F_{BD} \na_A \Fb^{AB}.
\end{align*}
\end{proof}

\begin{theorem}\label{main_com}
Suppose $(N,g)$ is a spacetime with Bondi-Sachs expansion \eqref{CWYexpansion}.  Further assuming that $(N,g)$ has vanishing linear momentum at null infinity \eqref{zeroLinMom} and positive Bondi mass, the center-of-mass integral at future null infinity has components
\begin{align}\label{comFinal}
\begin{split}
C^i &= \frac{1}{8\pi e} \int_{S^2} \mathring\nabla^{A}\tilde{X}^{i} \Bigg[ N_A - \frac{1}{4} C_{AB} \na_D C^{DB} - \frac{1}{16} \na_A \lt( C^{DE}C_{DE}\rt) \\
&\qquad\qquad\qquad\qquad\qquad\qquad\qquad\qquad - c \na_A m  + 2 \mathring\epsilon_{AB} (\na^B \cb) m \Bigg]\\
&\quad  + \frac{1}{8\pi e} \int_{S^2}  3 \tilde X^i cm - \frac{1}{4} \tilde X^i \na_A \Fb^{AB} \na^D \Fb_{DB}.
\end{split}
\end{align} 
In particular, if $C_{AB}$ is closed, then
\[C^i =\frac{1}{8\pi e} \int_{S^2} \mathring\nabla^{A}\tilde{X}^{i} \lt( N_A  - c \na_A m \rt) + 3 \tilde X^i cm, \]
thanks to \eqref{integral_identity_FF}. In addition, if $C_{AB}$ is closed and $m$ is angular independent, then
\[C^{i} = \frac{1}{8\pi e} \int_{S^2} \mathring\nabla^{A}\tilde{X}^{i}N_{A}.\]
\end{theorem} 
\begin{proof}
We simplify the last two terms in \eqref{pre_com}: 
\begin{align*}
&\int_{S^2} \tilde X^i \lt( \frac{1}{2} \na_A \lt( C^{AB} \na^D \Fb_{DB} \rt) - \frac{1}{4} C^{AB} \Fb_{AB} \rt) \\
&= \int_{S^2} \tilde X^i \cdot  \frac{1}{4} \na_A \lt( F^{AB} \na^D \Fb_{DB} + \Fb^{AB} \na^D \Fb_{DB} + \Fb^{AB} \na^D \Fb_{DB} \rt) \\
&\quad + \int_{S^2} \tilde X^i \lt( -\frac{1}{4} F^{AB} \Fb_{AB} - \frac{1}{8} \Fb^{AB} \Fb_{AB} \rt)\\
&= \int_{S^2}  \tilde X^i \lt( \frac{1}{4}\na_A\lt( C^{AB} \na^D C_{DB} \rt) - \frac{1}{8}C^{AB}C_{AB} \rt)\\
&= \int_{S^2} \na^A \tilde X^i \lt( -\frac{1}{4} C_{AB} \na_D C^{DB} - \frac{1}{16} \na_A \lt( C^{AB}C_{AB}\rt)\rt)  
\end{align*}
where we used \eqref{integral_identity_FFb} and \eqref{integral_identity_FbFb} in the first and \eqref{integral_identity_FF} in the second equation.
We thus have calculated \eqref{comFinal} assuming an observer $T_0 = (1,0,0,0)$.  More generally, the condition \eqref{zeroLinMom} allows for observers $T_{0,r}$ expanding according to \eqref{T0expansion}.  Following Definition \ref{qlCharge}, we write $T_{0,r} = A_{r}((1,0,0,0))$ for Lorentz transformations $A_{r}$ and measure the quasi-local center-of-mass \eqref{qlCOM} with respect to the Lorentz boost $A_{r}(K_{i,r})$, taking limits as $r$ goes to infinity to recover the components of the center-of-mass integral at future null infinity.

Owing to preservation under Lorentz transformations, the inner product expansion \eqref{innerProdExpansion} remains the same.  Likewise, the density terms in the expansion \eqref{densityFirst} and the current term $j_{A}^{(-1)}$ \eqref{jFirst} are unchanged.  Using the expansion of the Lorentz transformations $A_{r}$ \eqref{boostExpansion}, we have
\[ A_r(K_{i,r}) = x^i\partial_{t} + t\partial_{i} + \frac{1}{r}\left(x^{i} \sum_{j=1}^3 b_{j}\partial_j + tb_i\partial_t + t \sum_{j=1}^3 a_{ij}\partial_{j}\right) + O(r^{-2}),\]
from which application of the projection formula \eqref{projectionFormula} gives the expansion 
\begin{align*}
(A_{r}(K_{i,r}))^{T,B} &=\frac{1}{r}\left(-\mathring\nabla^{B}\left(\tilde{X}^{i}X^{0(0)}\right) + 2\mathring\nabla^{B}\tilde{X}^{i}X^{0(0)} + \tilde{X}^{i}b_j\mathring\nabla^{B}\tilde{X}^{j}\right) \\
&\quad + O(r^{-2}),
\end{align*}
differing from the earlier \eqref{projectionExpansion} in its final term.  The new terms in the center-of-mass component amount to
\begin{align*}
&\frac{1}{16\pi e}\int_{S^2} \tilde{X}^{i}b_{j}\mathring\nabla^{B}\tilde{X}^{j}\mathring\nabla^{D}\underline{F}_{DB}=0,
\end{align*}
by \eqref{jFirst} and Lemma \ref{xixj-co-closed} below. In this way, the earlier calculation of \eqref{comFinal} is preserved.
\end{proof}

\begin{lemma}\label{xixj-co-closed}
For any co-closed symmetric traceless two-tensor $\Fb_{AB},$ we have
\[ \int_{S^2} \tilde X^i \na^B \tilde X^j \na^D \Fb_{DB}=0.\]
\end{lemma}
\begin{proof}
Since closed and co-closed symmetric traceless two-tensors are orthogonal, we have
\begin{align*}
0 &= \frac{1}{2} \int_{S^2} \lt( \na^D \na^B (\tilde X^i \tilde X^j) - \mathring\Delta (\tilde X^i \tilde X^j) \mathring\sigma^{DB} \rt) \Fb_{DB}\\
&= \int_{S^2} \na^D \tilde X^i \na^B \tilde X^j \Fb_{DB} \\
&= - \int_{S^2} \tilde X^i \na^B \tilde X^j \na^D \Fb_{DB}.
\end{align*}
\end{proof}

\section{Example: Kerr Spacetime}

\subsection{Singular Bondi-Sachs Coordinates}

Barnich-Troessaert \cite{Barnich} provide a BMS coordinate system $(u,r,\theta,\phi)$ for the Kerr spacetime, such that 
\[g_{rr} = g_{r\theta} = g_{r\theta} = 0,\]
and the spherical part of the metric satisfies the determinant condition \eqref{areaNorm}
\[\det g_{AB} = r^4 \sin^2\theta.\]  
The remaining components expand as 
\begin{align}
\begin{split}
g_{uu} &= -1 + 2Mr^{-1} + O(r^{-2}),\\
g_{ur} &= -1 + a^2(\frac{1}{2}-\cos^2\theta)r^{-2} + O(r^{-3}),\\
g_{u\theta} &= \frac{a\cos\theta}{2\sin^2\theta} + \frac{a\cos\theta}{4}\left(8M + \frac{a}{\sin^3\theta}\right)r^{-1} + O(r^{-2}),\\
g_{u\phi} &= -2aM\sin^2\theta r^{-1} + O(r^{-2}),\\
g_{\theta\theta} &= r^2 + \frac{a}{\sin\theta} r + \frac{a^2}{2\sin^2\theta} + O(r^{-1}),\\
g_{\theta\phi} &= O(r^{-1}),\\
g_{\phi\phi} &= r^2\sin^2\theta - a\sin\theta r + \frac{a^2}{2} + O(r^{-1}),
\end{split}
\end{align}
where $M$ is the mass and $a$ is the angular velocity.  

Regarding the quantities appearing in the expansion \eqref{CWYexpansion}, the mass aspect $m = M$ is constant, while the shear tensor $C_{AB}$ has closed form, with components
\begin{align}\label{shearKerr}
\begin{split}
C_{\theta\theta} &= \frac{a}{\sin\theta},\\
C_{\theta\phi} &= 0,\\
C_{\phi\phi} &= -a\sin\theta,
\end{split}
\end{align}
and potentials
\begin{align}
\begin{split}
c &= -2a\sin\theta,\\
\underline{c} &= 0.
\end{split}
\end{align}

In addition, the angular momentum aspect $N_{A}$ has components
\begin{align}\label{angAspectKerr}
\begin{split}
N_{\theta} &= 3Ma\cos\theta,\\
N_{\phi} &= -3Ma\sin^2\theta.
\end{split}
\end{align}

We note that certain metric components above are singular.  The singular nature of the coordinate system is also apparent in the Bondi-Sachs data; in particular, the shear tensor $C_{AB}$ is singular.

As the mass aspect is constant, the condition \eqref{zeroLinMom} holds, yielding vanishing of the linear momentum at future null infinity.  That is, the spacetime is in center-of-mass frame, and solvability of the linearized optimal isometric embedding equation \eqref{firstOrder} is ensured.  In particular, we find
\[ X^{0(0)} = -\frac{1}{4}(\mathring\Delta +2)c.\]
Considering the next order in the optimal isometric embedding equation, we can directly compute vanishing of the term $S$ in the discussion following \eqref{secondOrder}, using the mass aspect, the angular momentum aspect \eqref{angAspectKerr}, the shear tensor \eqref{shearKerr}, and the form of $X^{0(0)}$ obtained by solving the linearized equation \eqref{firstOrder}.  As a consequence, we consider the observer $T_0 = (1,0,0,0)$, such that the $b_{i} = 0$ in the observer expansion \eqref{T0expansion}.

In calculating the angular momentum and center-of-mass integral at future null infinity, it is helpful to express the first eigenfunctions in spherical coordinates:
\begin{align}
\begin{split}
\tilde{X}^1 &= \sin\theta\cos\phi,\\
\tilde{X}^2 &= \sin\theta\sin\phi,\\
\tilde{X}^3 &= \cos\theta.
\end{split}
\end{align}

Owing to the simplicity of the Kerr spacetime, with constant mass aspect and closed shear tensor, we apply the special case appearing below our general formula for the quasi-local center-of-mass integral at null infinity \eqref{comFinal} to deduce
\[ C^{i} = \frac{1}{8\pi} \int_{S^2} \mathring\nabla^{A}\tilde{X}^{i}N_{A} = 0.\]

On the other hand, the special case appearing below our general formula for the angular momentum \eqref{angmomFinal} yields
\begin{align}
\begin{split}
J^1 &= \frac{1}{8\pi}\int_{S^2} \mathring\epsilon^{AB}\mathring\nabla_{B}\tilde{X}^1N_{A} = 0,\\
J^2 &= \frac{1}{8\pi}\int_{S^2}\mathring\epsilon^{AB}\mathring\nabla_{B}\tilde{X}^{2} N_{A} = 0,\\
J^3 &= \frac{1}{8\pi}\int_{S^2}\mathring\epsilon^{AB}\mathring\nabla_{B}\tilde{X}^{3} N_{A} = -Ma,
\end{split}
\end{align}
where
\[ \mathring\epsilon^{AB}\mathring\nabla_{B}\tilde{X}^3 = \partial_{\phi}\]
in spherical coordinates.  In this way, the calculations are precisely what we expect from the usual presentation of the Kerr spacetime.  

\subsection{Approximate Bondi-Sachs coordinates}
An approximate Bondi-Sachs coordinate system $(\bar{u}, \bar{x}, \bar{\theta}, \bar{\phi})$ is constructed in Chru\`{s}ciel-Jezierski-Kijowski \cite{CJK} with Bondi-Sachs data
\begin{align}
\begin{split}
m &= M,\\
C_{AB} &= 0,\\
N^{\bar{\theta}} &= 0,\\
N^{\bar{\phi}} &= -3Ma.
\end{split}
\end{align}
For more details on the authors' construction, see Appendix C.7 of \cite{CJK}.

The discussion in the previous subsection carries through largely unchanged; in particular, we have vanishing of linear momentum and solvability of the optimal isometric embedding equation, with $X^{0(0)} = 0$ and the observer $T_0 = (1, 0, 0, 0)$.

The center-of-mass calculation is even simpler than in the previous coordinate system, since the present coordinate system has a divergence-free angular momentum aspect.  As before, we find
\[ C^{i} = \frac{1}{8\pi} \int_{S^2} \mathring\nabla^{A}\tilde{X}^{i}N_{A} = 0.\]

Likewise, the angular momentum components are the same:
\begin{align}
\begin{split}
J^1 &= \frac{1}{8\pi}\int_{S^2} \mathring\epsilon^{AB}\mathring\nabla_{B}\tilde{X}^1N_{A} = 0,\\
J^2 &= \frac{1}{8\pi}\int_{S^2}\mathring\epsilon^{AB}\mathring\nabla_{B}\tilde{X}^{2} N_{A} = 0,\\
J^3 &= \frac{1}{8\pi}\int_{S^2}\mathring\epsilon^{AB}\mathring\nabla_{B}\tilde{X}^{3} N_{A} = -Ma.
\end{split}
\end{align}

\appendix
\section{Two Tensor Identities and Expansion of $V$}
In the appendix, we first derive some identities concerning the derivatives of closed and co-closed traceless symmetric two-tensors on $S^2$. Then we  prove the expansion of metric coefficient $V$ claimed in Proposition \ref{V}. 

\begin{proposition}\label{interchange}
Let $C_{AB}$ be a symmetric traceless 2-tensor on $(S^2,\mathring\sigma)$. Then
\begin{align}
\na_A C_{BD} - \na_B C_{AD} &= \na^E C_{BE} \mathring\sigma_{AD} - \na^E C_{AE} \mathring\sigma_{BD} \label{Codazzi}\\
\mathring\epsilon^{AB}\na_A C_{BD} &= \mathring\epsilon_{DA} \na_B C^{AB} \label{contract_codazzi}
\end{align}
\end{proposition}
\begin{proof}
The second identity follows by contracting the first one by $\mathring\epsilon^{AB}$. It suffices to verify the first identity for an orthornomal frame $e_1,e_2$. For $A = e_1, B =e_2, D=e_1$, we have
\begin{align*}
\na_1 C_{21} - \na_2 C_{11} = \na_1 C_{12} + \na_2 C_{22} = \na^E C_{E2};
\end{align*}
for $A = e_1, B= e_2, D=e_2$, we have
\begin{align*}
\na_1 C_{22} - \na_2 C_{12} = -\na_1 C_{11} - \na_2 C_{21} = -\na^E C_{E1}.
\end{align*}
This proves the assertion.
\end{proof}

\begin{proposition}
Let $C_{AB}$ be a traceless symmetric two-tensor on $S^2$. Then we have
\begin{align}
&\na^D \na_A C_{BD} + \na^D \na_B C_{AD} - \mathring\Delta C_{AB} = \mathring\sigma_{AB} \na^D \na^E C_{DE} + 2 C_{AB},\label{interchangeThree}
\end{align}
\end{proposition}
\begin{proof}
We work at the potential level and compute
\begin{align*}
\na^D \na_A C_{BD} &= \na_A \na^D C_{BD} - R^{D\;\;E}_{\;\;A\;\;B}C_{ED} - R^{D\;\;E}_{\;\;A\;\;D}C_{BE}\\
&= \frac{1}{2}\na_A \na_B(\mathring\Delta + 2)c + \frac{1}{2} \epsilon_B^{\;\;D} \na_A\na_D(\Delta + 2)\cb + 2C_{AB}.
\end{align*}
For a function $f$, we have
\[ \mathring\Delta \na_A\na_B f = \na_A\na_B (\mathring\Delta +2) f + 2 \na_A\na_B f -2 \mathring\Delta f \mathring\sigma_{AB}. \]
Hence,
\begin{align*}
\mathring\Delta C_{AB} &= \na_A\na_B (\mathring\Delta + 2) c - \frac{1}{2}\mathring\Delta (\mathring\Delta + 2)c \mathring\sigma_{AB} \\
&\quad + \frac{1}{2} \lt( \mathring\epsilon_A^{\;\;D} \na_D\na_B (\mathring\Delta + 2)\cb + \mathring\epsilon_B^{\;\;D} \na_D\na_A (\mathring\Delta + 2)\cb \rt) + 2 C_{AB}.
\end{align*}
Putting these together, we obtain
\begin{align*}
\na^D\na_A C_{BD} + \na^D\na_B C_{AD} - \mathring\Delta C_{AB} &= \frac{1}{2}\mathring\Delta(\mathring\Delta + 2)c \mathring\sigma_{AB} + 2 C_{AB}\\
&= \mathring\sigma_{AB} \na^D\na^E C_{DE} + 2C_{AB}
\end{align*}  
\end{proof}

\begin{lemma}\label{magical identity} 
We have the following identity:
\begin{equation}
\frac{1}{2} R^{(2)} + \frac{1}{4} \na_A \lt( C^{AB} \na^D C_{BD} \rt) + \frac{1}{16} \mathring{\Delta}(C_{DE}C^{DE})=0,
\end{equation}
where
\begin{align}\label{scalarCurv}
\begin{split}  
R^{(2)} &:= \frac{1}{2}\left(C_{DE}C^{DE}\right) +\frac{1}{2}\mathring\nabla_{A}C_{BD}\mathring\nabla^{B}C^{AD} + \frac{1}{4}\mathring\Delta\left(C_{DE}C^{DE}\right)\\
&- \mathring\nabla_{A}\left(C^{AB}\mathring\nabla^{D}C_{BD} + C^{BD}\mathring\nabla_{B}C_{D}^{A}\right) - \frac{1}{4}\left(\mathring\nabla_{A}C_{BD}\mathring\nabla^{A}C^{BD}\right).
\end{split}
\end{align}
\end{lemma}
\begin{proof}
By Proposition \ref{interchange}, we have
\begin{align}
\na_A C_{BD} \na^A C^{BD} = \na_A C_{BD} \na^B C^{AD} + \na^B C_{AB} \na_D C^{AD}.
\end{align}
Indeed, \begin{align*}
&\na_A C_{BD} \na^A C^{BD} \\
&= \na_A C_{BD} \lt( \na^B C^{AD} - \frac{1}{2} \mathring\epsilon^{AB} \mathring\Delta \na^D \cb + \frac{1}{2} \mathring\epsilon^{DA} \na^B \cb + \na^E F_{BE} \mathring\sigma_{AD} \rt)\\
&= \na_A C_{BD} \na^B C^{AD} + \frac{1}{4} \lt( \na_D (\mathring\Delta + 2) \cb - \mathring\epsilon_{DB} \na^B(\mathring\Delta+2) c \rt) \mathring\Delta \na^D \cb\\
&\quad + \frac{1}{4} \lt( \na_B (\mathring\Delta + 2) \cb\ - \mathring\epsilon_{BD} \na^D(\mathring\Delta + 2)c \rt) \na^B \cb\\
&\quad + \frac{1}{2} \na^E F_{BE} \lt( \na^B(\mathring\Delta + 2) c + \mathring\epsilon^{BD} \na_D (\mathring\Delta + 2) \cb \rt) \\
&= \na_A C_{BD} \na^B C^{AD} + \frac{1}{4} \na_{A}(\mathring\Delta + 2) \cb\na^{A}(\mathring\Delta +2)\cb\\
& + \frac{1}{2} \mathring\epsilon^{BD} \na_B (\mathring\Delta + 2) c \na_D (\mathring\Delta + 2) \cb + \frac{1}{4} \na_{A}(\mathring\Delta +2) c\na^{A}(\mathring\Delta+2)c\\\
&= \na_A C_{BD} \na^BC^{AD} + \na^B C_{AB} \na_D C^{AD}.
\end{align*}
Moreover, contracting \eqref{interchangeThree} with $C^{AB}$, we obtain
\[C^{AB} \mathring\Delta C_{AB} = 2 C^{AB} \na^D\na_A C_{BD} - 2(C_{DE}C^{DE}).\]
We deduce
\begin{align*}
&R^{(2)} + \frac{1}{2} \na_A \lt( C^{AB} \na^D C_{BD}\rt) + \frac{1}{8} \mathring\Delta(C_{DE}C^{DE})\\
&= \frac{1}{2}(C_{DE}C^{DE}) + \frac{3}{4} \lt( \mathring\Delta C_{AB} C^{AB} + \na_A C_{BD} \na^A C^{BD} \rt) + \frac{1}{4} \na_A C_{BD} \na^B C^{AD}\\
&\quad- \frac{1}{4} \na^B C_{AB} \na_D C^{AD} - \frac{1}{2} \na_A \lt( C^{AB} \na^D C_{BD} \rt) - \na_A \lt( C^{BD} \na_B C^A_D \rt)\\
&= \frac{1}{2}(C_{DE}C^{DE})+ \frac{3}{2} C^{AB} \na^D \na_A C_{BD} - \frac{3}{2}(C_{DE}C^{DE}) + \na_A C_{BD} \na^B C^{AD}\\
& + \frac{1}{2} \na^B C_{AB} \na_D C^{AD}  - \frac{1}{2} \na_A \lt( C^{AB} \na^D C_{BD} \rt) - \na_A \lt( C^{BD} \na_B C^A_D \rt)\\
&= - (C_{DE}C^{DE}) + \frac{1}{2} C^{AB} \na^D \na_A C_{BD} -\frac{1}{2} C^{AB} \na_A \na^D C_{BD} \\
&= 0. 
\end{align*}
\end{proof}

We are ready to simplify the expansion of $V$.
\begin{proof}[Proof of Proposition \ref{V}] Per M\"{a}dler-Winicour \cite{MW}, we have the initial data equation
\begin{align}
\begin{split}
\partial_{r}(rV) &= \frac{r^2}{2}UR(\sigma) - \frac{r^2}{2}\Delta U + \frac{r^2}{4U}\nabla_A U\nabla^A U\\
& -\frac{1}{2r^2}\nabla_{A}\left(\partial_{r}(r^4W^{A})\right) - \frac{r^2}{4U}\sigma_{AB}\left(\partial_{r}W^{A}\right)\left(\partial_{r}W^{B}\right),
\end{split}
\end{align}
with $R(\sigma)$ the scalar curvature of $\sigma$.  

It is well-known that the scalar curvature expands as
\[R(\sigma) = \frac{2}{r^2} + \frac{1}{r^3}\mathring\nabla^{D}\mathring\nabla^{E}C_{DE} + \frac{1}{r^4}R^{(2)} + O(r^{-5}). \]
See \cite[Proposition 4]{BM} for example.

Recall
\eqref{CWYexpansion}
\begin{align*}
\begin{split}
W^{A} &= \frac{1}{r^2} W^{(-2)A} + \frac{1}{r^3} W^{(-3)A} + O(r^{-4}) \\
&= \frac{1}{2r^2}\mathring\nabla^{D}C_{D}^{A} + \frac{1}{r^3}\left(\frac{2}{3}N^{A} - \frac{1}{16}\mathring\nabla^{A}\left(C_{DE}C^{DE}\right) - \frac{1}{2}C^{A}_{B}\mathring\nabla^{D}C_{D}^{B}\right) \\
&\quad + O(r^{-4})
\end{split}
\end{align*}
and we get \begin{align*}
\begin{split}
V &= 1 - \frac{2m}{r} - \frac{1}{2r^2}R^{(2)} + \frac{1}{16r^2}\left(C_{DE}C^{DE}\right) - \frac{1}{32r^2}\mathring\Delta\left(C_{DE}C^{DE}\right)\\
&\quad +\frac{1}{4r^2}\left(\frac{4}{3}\mathring\nabla^{A}N_{A} - \frac{1}{8}\mathring\Delta\left(C_{DE}C^{DE}\right)- C_{AB}\mathring\nabla^{A}\mathring\nabla^{D}C_{D}^{B} \right) + O(r^{-3}).
\end{split}
\end{align*}
By Lemma \ref{magical identity}, we get
\begin{align*}
V &= 1 - \frac{2m}{r}\\
&\quad + \frac{1}{4r^2} \na_A C^{AB} \na^D C_{BD} + \frac{1}{16r^2} C_{DE}C^{DE} + \frac{1}{3r^2} \na^A N_A + O(r^{-3}).
\end{align*}
\end{proof}

\section{Decomposition of symmetric traceless 2-tensors on $S^2$}
\begin{theorem}
Let $C_{AB}$ be a symmetric traceless 2-tensor on $(S^2,\mathring\sigma)$ that is divergence-free, $\na^A C_{AB}=0.$ Then $C_{AB}=0.$
\end{theorem}
\begin{proof}
The assertion actually holds for any metric on $S^2,$ based on the fact that $S^2$ has no nontrivial holomorphic 1-form. For standard metric $\mathring\sigma,$ we present an elementary proof following M\"{a}dler-Winicour \cite[page 13-14]{MW}.

Let $Y^A = \mathring\epsilon^{AB} \na_B g$ where $g$ is a $-2$ eigenfunction of $(S^2,\mathring\sigma).$  By the identity $\na_A \na_B g = - g \mathring\sigma_{AB}$, we have $\na^B (Y^A C_{AB})=0$ and hence there exists a function $f$ such that $Y^A C_{AB} = \mathring\epsilon_{BD} \na^D f.$ We compute
\begin{align*}
\int_{S^2} (Y^A C_{AB})(Y^D C_D^B) &= \int_{S^2} \mathring\epsilon^{AE} \na_E g C_{AB} \mathring\epsilon^{BD} \na_D f \\
&= - \int_{S^2} \mathring\epsilon^{AE} \na_E g \na_D C_{AB} \mathring\epsilon^{BD} f\\
&= - \int_{S^2} (\mathring\sigma^{AB}\mathring\sigma^{DE} - \mathring\sigma^{AD}\mathring\sigma^{BE}) \na_E g \na_D C_{AB} f\\
&=0
\end{align*}
where we used the identity $\mathring\epsilon^{AB} \mathring\epsilon_{BD} = -\delta^A_D$ in the second equality. Thus $Y^A C_{AB}=0.$ Since for every point $p$ on $S^2$ and tangent vector $v \in T_p S^2,$ there exists $g$ such that $Y^A = v,$ we conclude that $C_{AB}=0.$
\end{proof}
\begin{theorem}
Let $C_{AB}$ be a symmetric traceless 2-tensor on $(S^2,\mathring\sigma).$ Then \begin{align*}
C_{AB}(u,x^{A}) &= \left(\mathring\nabla_{A}\mathring\nabla_{B} - \frac{1}{2}\mathring\sigma_{AB}\mathring\Delta\right)c(u,x^{D})\\
&+ \frac{1}{2}\left(\mathring{\epsilon}_{AD}\mathring\nabla^{D}\mathring\nabla_{B} + \mathring{\epsilon}_{BD}\mathring\nabla^{D}\mathring\nabla_{A}\right)\underline{c}(u,x^{D}),
\end{align*}
with scalar potentials $c(u,x^{D})$ and $\underline{c}(u,x^{D})$ and $\mathring\epsilon_{AB}$ the area form of the standard unit sphere.
\end{theorem}
\begin{proof}
Let \[ \na^A C_{AB} = \na_B f + \mathring\epsilon_{BD} \na^D g\]
be the Hodge decomposition of 1-form $\na^A C_{AB}$. We decompose $f$ and $g$ into spherical harmonics $f = f_{l=1} + f_{l \ge 2}$ and $g = g_{l=1} + g_{l \ge 2}.$ We first show that $f_{l=1} = g_{l=1}=0$. Indeed, using $\na_A \na_B f_{l=1} = - f_{l=1} \mathring\sigma_{AB}$, we integration by parts to get
\[ 0=\int_{S^2} \na^A C_{AB} \na^B f_{l=1} = \int_{S^2} |\na f_{l=1}|^2. \]
Similarly, we have
\[ 0 = \int_{S^2} \na^A C_{AB} \mathring\epsilon^{BD} \na_D g_{l=1} = \int_{S^2} |\na g_{l=1}|^2, \]
where we used the identity $\mathring\epsilon_{BD} \mathring\epsilon^{BE} = \delta_D^{\;\;E}.$ 

Since $f$ and $g$ both support in the $l \ge 2$ modes, there exist potentials $c$ and $\cb$ such that $\frac{1}{2}(\mathring\Delta + 2) c = f$ and $\frac{1}{2} (\mathring\Delta + 2) \cb = g.$ Direct computation shows that the symmetric traceless 2-tensor
\begin{align*}
\hat C_{AB} &= \left(\mathring\nabla_{A}\mathring\nabla_{B} - \frac{1}{2}\mathring\sigma_{AB}\mathring\Delta\right)c\\
&+ \frac{1}{2}\left(\mathring{\epsilon}_{AD}\mathring\nabla^{D}\mathring\nabla_{B} + \mathring{\epsilon}_{BD}\mathring\nabla^{D}\mathring\nabla_{A}\right)\underline{c}.
\end{align*}
satisfies $\na^A \hat C_{AB} = \na_B f + \mathring\epsilon_{BD} \na^D g = \na^A C_{AB}.$ By the previous theorem, we get $C_{AB} = \hat C_{AB}.$
\end{proof}

\bibliographystyle{plain}
\bibliography{AngMomCOMBib}

\begin{thebibliography}{10}

\bibitem{LIGO}
B.~P. Abbott and et~al.
\newblock Observation of gravitational waves from a binary black hole merger.
\newblock {\em Phys. Rev. Lett.}, 116(6):061102, 16, 2016.
\newblock Authors include B. C. Barish, K. S. Thorne and R. Weiss.

\bibitem{Ashtekar-Streubel}
A.~Ashtekar and M.~Streubel.
\newblock Symplectic geometry of radiative modes and conserved quantities at
  null infinity.
\newblock {\em Proc. Roy. Soc. London Ser. A}, 376(1767):585--607, 1981.

\bibitem{ADK}
Abhay Ashtekar, Tommaso De~Lorenzo, and Neev Khera.
\newblock Compact binary coalescences: the subtle issue of angular momentum.
\newblock {\em Phys. Rev. D}, 101(4):044005, 18, 2020.

\bibitem{Ashtekar-Winicour}
Abhay Ashtekar and Jeffrey Winicour.
\newblock Linkages and {H}amiltonians at null infinity.
\newblock {\em J. Math. Phys.}, 23(12):2410--2417, 1982.

\bibitem{Bar-Tro}
Glenn Barnich and C\'{e}dric Troessaert.
\newblock Aspects of the {BMS}/{CFT} correspondence.
\newblock {\em J. High Energy Phys.}, (5):062, 36, 2010.

\bibitem{Barnich}
Glenn Barnich and C\'{e}dric Troessaert.
\newblock B{MS} charge algebra.
\newblock {\em J. High Energy Phys.}, (12):105, 22, 2011.

\bibitem{BVM}
H.~Bondi, M.~G.~J. van~der Burg, and A.~W.~K. Metzner.
\newblock Gravitational waves in general relativity. {VII}. {W}aves from
  axi-symmetric isolated systems.
\newblock {\em Proc. Roy. Soc. Ser. A}, 269:21--52, 1962.

\bibitem{Bramson}
B.~D. Bramson.
\newblock Relativistic angular momentum for asymptotically flat
  {E}instein-{M}axwell manifolds.
\newblock {\em Proc. Roy. Soc. (London) Ser. A}, 341:463--490, 1974/75.

\bibitem{BM}
S.~Brendle and F.~C. Marques.
\newblock Scalar curvature rigidity of geodesic balls in {$S^n$}.
\newblock {\em J. Differential Geom.}, 88(3):379--394, 2011.

\bibitem{CKWWY}
P.-N Chen, J.~Keller, M.-T. Wang, Y.-K. Wang, and S.-T. Yau.
\newblock Evolution of angular momentum and center of mass at null infinity.
\newblock {\em To appear in Comm. Math. Phys.}

\bibitem{CWY}
Po-Ning Chen, Mu-Tao Wang, and Shing-Tung Yau.
\newblock Conserved quantities in general relativity: from the quasi-local
  level to spatial infinity.
\newblock {\em Comm. Math. Phys.}, 338(1):31--80, 2015.

\bibitem{CWY_null}
PoNing Chen, Mu-Tao Wang, and Shing-Tung Yau.
\newblock Evaluating quasilocal energy and solving optimal embedding equation
  at null infinity.
\newblock {\em Comm. Math. Phys.}, 308(3):845--863, 2011.

\bibitem{CK}
D.~Christodoulou and S.~Klainerman.
\newblock {\em {The Global Nonlinear Stability of the {M}inkowski Space}},
  volume~41 of {\em Princeton Mathematical Series}.
\newblock Princeton University Press, Princeton, NJ, 1993.

\bibitem{CJM}
P.~T. Chru\'{s}ciel, J.~Jezierski, and M.~A.~H. MacCallum.
\newblock Uniqueness of the {T}rautman-{B}ondi mass.
\newblock {\em Phys. Rev. D (3)}, 58(8):084001, 16, 1998.

\bibitem{CJK}
Piotr~T. Chru\'{s}ciel, Jacek Jezierski, and Jerzy Kijowski.
\newblock {\em Hamiltonian field theory in the radiating regime}, volume~70 of
  {\em Lecture Notes in Physics. Monographs}.
\newblock Springer-Verlag, Berlin, 2002.

\bibitem{Dray}
Tevian Dray.
\newblock Momentum flux at null infinity.
\newblock {\em Classical Quantum Gravity}, 2(1):L7--L10, 1985.

\bibitem{Dray-Streubel}
Tevian Dray and Michael Streubel.
\newblock Angular momentum at null infinity.
\newblock {\em Classical Quantum Gravity}, 1(1):15--26, 1984.

\bibitem{Flanagan-Nichols}
\'{E}anna~\'{E}. Flanagan and David~A. Nichols.
\newblock Conserved charges of the extended {B}ondi-{M}etzner-{S}achs algebra.
\newblock {\em Phys. Rev. D}, 95(4):044002, 19, 2017.

\bibitem{HPS}
Stephen~W. Hawking, Malcolm~J. Perry, and Andrew Strominger.
\newblock Superrotation charge and supertranslation hair on black holes.
\newblock {\em J. High Energy Phys.}, (5):161, front matter+32, 2017.

\bibitem{Horowitz-Perry}
Gary~T. Horowitz and Malcolm~J. Perry.
\newblock Gravitational energy cannot become negative.
\newblock {\em Phys. Rev. Lett.}, 48(6):371--374, 1982.

\bibitem{MW}
T.~{M{\"a}dler} and J.~{Winicour}.
\newblock {Bondi-Sachs Formalism}.
\newblock {\em Scholarpedia}, 11(12):33528, 2016.

\bibitem{Moreschi04}
Osvaldo~M. Moreschi.
\newblock Intrinsic angular momentum and centre of mass in general relativity.
\newblock {\em Classical Quantum Gravity}, 21(23):5409--5425, 2004.

\bibitem{Pen_quasilocal}
R.~Penrose.
\newblock Quasilocal mass and angular momentum in general relativity.
\newblock {\em Proc. Roy. Soc. London Ser. A}, 381(1780):53--63, 1982.

\bibitem{Prior}
C.~R. Prior.
\newblock Angular momentum in general relativity. {I}. {D}efinition and
  asymptotic behaviour.
\newblock {\em Proc. Roy. Soc. London Ser. A}, 354(1679):379--405, 1977.

\bibitem{R_thesis}
Anthony Rizzi.
\newblock Angular momentum in general relativity.
\newblock {\em Thesis, Princeton University}, 1997.

\bibitem{R_def}
Anthony Rizzi.
\newblock Angular momentum in general relativity: a new definition.
\newblock {\em Phys. Rev. Lett.}, 81(6):1150--1153, 1998.

\bibitem{Sachs}
R.~K. Sachs.
\newblock Gravitational waves in general relativity. {VIII}. {W}aves in
  asymptotically flat space-time.
\newblock {\em Proc. Roy. Soc. Ser. A}, 270:103--126, 1962.

\bibitem{Schoen-Yau}
Richard Schoen and Shing~Tung Yau.
\newblock Proof that the {B}ondi mass is positive.
\newblock {\em Phys. Rev. Lett.}, 48(6):369--371, 1982.

\bibitem{Shaw}
William~T. Shaw.
\newblock Symplectic geometry of null infinity and two-surface twistors.
\newblock {\em Classical Quantum Gravity}, 1(4):L33--L37, 1984.

\bibitem{Wald-Zoupas}
Robert~M. Wald and Andreas Zoupas.
\newblock General definition of ``conserved quantities'' in general relativity
  and other theories of gravity.
\newblock {\em Phys. Rev. D (3)}, 61(8):084027, 16, 2000.

\bibitem{Winicour}
J.~Winicour and L.~Tamburino.
\newblock Lorentz-covariant gravitational energy-momentum linkages.
\newblock {\em Phys. Rev. Lett.}, 15:601--605, 1965.

\end{thebibliography}

\end{document}